\documentclass[12pt,a4paper,reqno]{amsart}
\usepackage[dvipsnames]{xcolor}
\usepackage{amsmath,amssymb,graphics,epsfig,color,enumerate,mathrsfs}
\usepackage[normalem]{ulem}
\usepackage{dsfont}  \usepackage{verbatim}
\definecolor{refkey}{gray}{.75}
\definecolor{labelkey}{gray}{.5}

\newtheorem{Theorem}{Theorem}[section]
\newtheorem{Fact}[Theorem]{Fact}

\newtheorem{Lemma}[Theorem]{Lemma}
\newtheorem{Proposition}[Theorem]{Proposition}
\newtheorem{Corollary}[Theorem]{Corollary}
\newtheorem{Remark}[Theorem]{Remark}

\newtheorem{Definition}[Theorem]{Definition}
\newtheorem{Warning}[Theorem]{Warning}

 \definecolor{darkgreen}{rgb}{0,0.4,0}

\definecolor{light}{gray}{0.9}

\newcommand{\cA}{\ensuremath{\mathcal A}}
\newcommand{\cB}{\ensuremath{\mathcal B}}
\newcommand{\cC}{\ensuremath{\mathcal C}}
\newcommand{\cD}{\ensuremath{\mathcal D}}
\newcommand{\cE}{\ensuremath{\mathcal E}}

\newcommand{\cG}{\ensuremath{\mathcal G}}
\newcommand{\cH}{\ensuremath{\mathcal H}}

\newcommand{\cL}{\ensuremath{\mathcal L}}

\newcommand{\cN}{\ensuremath{\mathcal N}}

\newcommand{\cS}{\ensuremath{\mathcal S}}

\newcommand{\cV}{\ensuremath{\mathcal V}}

\newcommand{\cX}{\ensuremath{\mathcal X}}

\newcommand{\bbB}{{\ensuremath{\mathbb B}} }

\newcommand{\bbE}{{\ensuremath{\mathbb E}} }

\newcommand{\bbG}{{\ensuremath{\mathbb G}} }

\newcommand{\bbI}{{\ensuremath{\mathbb I}} }

\newcommand{\bbN}{{\ensuremath{\mathbb N}} }

\newcommand{\bbP}{{\ensuremath{\mathbb P}} }
\newcommand{\bbQ}{{\ensuremath{\mathbb Q}} }
\newcommand{\bbR}{{\ensuremath{\mathbb R}} }

\newcommand{\bbV}{{\ensuremath{\mathbb V}} }

\newcommand{\bbZ}{{\ensuremath{\mathbb Z}} }

\newcommand{\lrgh}{\longleftrightarrow}

\newcommand{\rgh}{\rightarrow}

\newcommand{\mamma}{\magenta{RN}$[\b]$}

\newcommand{\mammaL}{\magenta{RN}$[\b,\ell]$}

\newcommand{\V}{V_{\b,\ell}^\o}

\newcommand{\giggio}{$\cG[\z, \b]$}

\newcommand{\giggioest}{$\cG[\z,\b](\o)$}

\let\a=\alpha \let\b=\beta   \let\d=\delta  \let\e=\varepsilon
 \let\g=\gamma     \let\k=\kappa  \let\l=\lambda
      \let\o=\omega    \let\p=\pi  
\let\r=\rho  \let\s=\sigma \let\t=\tau   
 \let\x=\xi \let\z=\zeta
\let\D=\Delta     \let\L=\Lambda 
\let\O=\Omega      
\newcommand{\rosso}{\textcolor{black}} 
\newcommand{\vvv}{\textcolor{black}} 
\newcommand{\blu}{\textcolor{black}}
\newcommand{\orc}{\textcolor{black}} 

\newcommand{\rot}{\textcolor{black}} 
\newcommand{\magenta}{\textcolor{black}} 
\newcommand{\verde}{\textcolor{black}}

\newcommand{\da}{\downarrow}

\newcommand{\be}{\begin{equation}}
\newcommand{\en}{\end{equation}}

\newcommand{\gino}{${\rm PPP}[\rho]$, ${\rm PPP}[\rho,\nu]$}

\newcommand{\cv}{\rosso{\cV_\alpha}}
\newcommand{\cvp}{\rosso{\cV^+_\alpha}}
\newcommand{\vpz}{\vvv{\bbP^{\hspace{0.013cm}0} }}
\newcommand{\hvpz}{\vvv{{\hat \bbP}^{\hspace{0.013cm}0} }}
\newcommand{\bvpz}{\vvv{{\bar\bbP}^{\hspace{0.013cm}0} }}

\newcommand{\vqz}{\vvv{\bbQ^{\hspace{0.013cm}0} }}

\newcommand{\pn}{\rosso{{\hat\bbP}^{\hspace{0.03cm}n}}}

\author[A.~Faggionato]{Alessandra Faggionato}
\address{Alessandra Faggionato.
  Dipartimento di Matematica, Universit\`a di Roma `La Sapienza'
  P.le Aldo Moro 2, 00185 Roma, Italy}
\email{faggiona@mat.uniroma1.it}

\title[Mott's  law]{Mott's law for the  \magenta{v.r.h.} random resistor network and for  Mott's random walk}

\begin{document}

\begin{abstract}
  Mott's  variable range hopping (v.r.h.) is  the phonon-induced hopping of electrons in disordered solids (\rot{such} as doped semiconductors) within the regime of strong Anderson localization. It was introduced by N.~Mott  to explain the anomalous low temperature conductivity decay in dimension $d\geq 2$, corresponding  now to the  so called Mott's law. We provide a rigorous derivation of this Physics law for two effective models of Mott v.r.h.: the  random resistor network for \magenta{v.r.h. of \cite[Section~IV]{AHL}} and Mott's random walk. We also determine the  constant multiplying the power of the inverse temperature in the exponent in Mott's law, which was an open problem also on a  heuristic level. 
  

\smallskip


\smallskip


\end{abstract}

\maketitle

\section{Introduction}

Mott's  variable range hopping is a mechanism of phonon--assisted electron transport taking place in amorphous solids (\rot{such as} doped semiconductors)  in the regime of strong Anderson localization \cite{AHL,MA,MD,POF,Sa,SE}. It has been introduced by N.~Mott in order to explain the anomalous 
 decay of the conductivity  at low temperature \cite{Mott,Mott_Nob}.  
 
Calling $x$ the  impurity positions  in the doped semiconductor, the electron Hamiltonian has exponentially localized quantum eigenstates with localization centers $x$ and corresponding energy $E_{x}$  close to the Fermi level, set equal to zero in what follows. At low temperature phonons  induce transitions between the localized eigenstates, the rate of which can be calculated by the Fermi golden rule \cite{MA,SE}. In the simplification of spinless electrons,   the resulting rate for an electron to  hop from $x$ to the unoccupied site $y$ is then given by  (cf. \cite[Eq.~(3.6)]{AHL}) 
\be\label{tunnel} \exp\Big\{-\frac{2}{\g}  |x-y| - \beta \{ E_{x} -E_{y} \}_+\Big\} \,.\en
In \eqref{tunnel} $\g$ is the localization length, $\beta=1/k T $ is the inverse temperature and $\{a\}_+:= \max\{0,a\}$.

The above set \rosso{of impurity locations} can be modeled by a   random simple point process (SPP), marked by  random variables $E_x$ (called energy marks) which are usually taken i.i.d. with some common distribution $\nu$. The physically relevant distributions for inorganic media are of the form  $\nu(dE)= c\, |E|^{\alpha} dE$ with  finite support  $[-A,A]$  for some exponent $\a\geq 0$ \cite{SE}. Mott's law \cite{Mott,Mott_Nob,MD,SE} then  predicts  that, for $d\geq 2$,  the \rot{conductivity} matrix $\s(\b)$ of the medium decays to zero when  $\b \to+ \infty$ as 
\be\label{mottino}
\s(\b) \approx  A(\b)\exp \bigl\{ -\k \, \beta ^\frac{\a+1}{\a+1+d}\bigr\}\,,
 \en
 where the prefactor matrix $A(\b)$ exhibits  a negligible $\b$--dependence (to simplify here the exposition, we restrict just now to isotropic media).
 Strictly speaking, Mott derived  the above asymptotics for $\a=0$, while  Efros and Shklovskii derived its extension. 
  For $d=1$, \rot{according to the marked SPP,  the conductivity is zero or} presents  an Arrenhius--type decay  as $\b \to+ \infty$  for all $\a\geq 0$ \cite{Kur}, i.e.\ 
 \be\label{mottino1}
 \s(\b)\approx  A(\b)\exp \bigl\{ -\k \b\}\,.
 \en

Due to   localization one can treat the above electron conduction by a hopping process of classical particles (see also \cite{ABS,BRSW}), thus \rot{leading   to} a  simple exclusion process due to the  Pauli blocking. \rot{This exclusion process takes place in a random environment, is non-gradient with non-symmetric rates, and has infinite range, thus making it technically challenging.}
  The reversible measure of the exclusion process is  the Fermi-Dirac distribution, i.e. the Bernoulli product probability measure  such that the probability of having a particle at $x$ is proportional to $e^{-\b (E_x-c)}$, $c\in \bbR$ being the chemical potential. \rot{Since we have zero Fermi energy, the relevant $c$ is given by $c=0$}. Effective simplified models   are  given by the  random resistor network  for v.r.h. \magenta{(cf.~\cite[Section~IV]{AHL} and \cite{POF,Sa,SE})}   and  by Mott's random walk \cite{FSS}.
The \magenta{v.r.h.} random resistor network has nodes $x$ and,   between any  pair of nodes $x\not=y$, it has  an electrical filament of  conductivity 
\be\label{condu}
c_{x,y}:=\exp\Big\{ - \frac{2}{\gamma} |x-y| -\frac{\b}{2} ( |E_x|+ |E_y|+ |E_x-E_y|) \Big\}\,.
\en
Mott's random walk is the continuous--time random walk with \rosso{states given by the impurity locations} and probability rate for a jump from $x$ to $y$ given by \eqref{condu}. We point out that the r.h.s. of \eqref{condu}   corresponds to the leading term for $\b$ large of \eqref{tunnel} multiplied by the probability in the Fermi-Dirac distribution that $x$ and $y$ are, respectively,   occupied and unoccupied  by an electron 
(see \rot{\cite[Section~III]{AHL} and in particular} \cite[Eq.~(3.7)]{AHL}). A more detailed discussion on the \magenta{v.r.h.} resistor network and Mott's random walk  is provided in Sections \ref{ssec_MA} and \ref{sec_mott_rw}.

The original  derivation of the  laws \eqref{mottino} and \eqref{mottino1} is rather heuristic.  More robust    arguments  have been proposed in the Physics  literature (see \cite{AHL,Po,POF,Sa,SE}), in particular for the case $d\geq 2$.  Rigorous bounds in agreement with \eqref{mottino} and \eqref{mottino1} have beed derived in \cite{CF1,Fhom2,FM,FSS} (see Section~\ref{sec_mott_rw}). The characterization of the constant $\k$ was an open problem, even at heuristic level (see \cite[Chapter~5]{POF}).

\smallskip
\rot{Let us describe our main contributions presented in Sections~\ref{MR_PPP} and~\ref{MR_univ}: Theorem~\ref{teo1}, Corollary~\ref{luna}, Theorem~\ref{parataUB}, Theorem~\ref{parataLB}, Corollary~\ref{immenso} and Theorems~\ref{paradiso}, \ref{supersanto}, \ref{santantonio}.}
  First  of all we point out that, by means of homogenization theory,  under very general assumptions in \cite{Fhom2}
we proved  that the conductivity matrix $\s(\b)$ of the \magenta{v.r.h.} random resistor network equals the asymptotic diffusion matrix $D(\b)$  of Mott's random walk times the intensity $\rho$ of the SSP  (i.e. the mean impurity density). Hence, all results presented below for $\s(\b)$ hold also for $D(\b)$. \magenta{At cost of a rescaling   but w.l.o.g., in \eqref{condu} we can take $\g=2$ and we can replace $\b/2$ by $\b$, hence in the rest $c_{x,y}:=\exp\{ - |x-y| -\b ( |E_x|+ |E_y|+ |E_x-E_y|) \}$.}

\smallskip

 In Section \ref{MR_PPP} we  take $d\geq 2$ and  consider as  marked SPP  a homogeneous Poisson point process (\magenta{briefly}, PPP)  with intensity $\rho$ and i.i.d. energy marks   with common distribution $\nu$ which, in some neighborhood of the origin,
 has density proportional to $|E|^\a$ or $E^\a \mathds{1}(E\geq 0)$, where $\a\geq 0$. In particular, for some $C_{\rot{\nu}}>0$, near the origin $\nu$
    has   the form $(\a+1) C_{\rot{\nu}}^{-\a-1}2^{-1} |E|^\a dE$ (case $\nu\in \cv$)  or $(\a+1)C_{\rot{\nu}}^{-\a-1} E^\a \mathds{1}(E\geq 0) dE$  (case $\nu\in \cvp$).

$\bullet$ {\bf Theorem~\ref{teo1}, Corollary~\ref{luna}.} We first consider 
  the critical conductance $c_c(\beta, \rho,\nu)$. According to \cite{AHL}  $c_c(\beta, \rho,\nu)$ is such that the graph obtained from the \magenta{v.r.h.} resistor network by keeping only the filaments  $\{x,y\}$ 
with conductivity $c_{x,y}>C$ a.s. does not percolate if $C>c_c(\beta, \rho,\nu)$, while it  a.s. percolates if $c_c(\beta, \rho,\nu)>C$.  In Theorem \ref{teo1} and its Corollary \ref{luna} we show that the limit 
\[ \chi(\rho, \nu):=\lim _{\b\to +\infty} \b^{- \frac{\a+1}{\a+1+d} } \ln c_c(\b,\rho,\nu )\,\] exists and compute it \magenta{(see \eqref{eccolino})}. 

$\bullet$ {\bf Theorems~\ref{parataUB} and \ref{parataLB}, Corollary~\ref{immenso}.}
We then prove Mott's law \eqref{mottino} for $\nu\in \cvp$ (\rot{cf.~ Corollary  \ref{immenso}}), showing \magenta{that} the constant $\k$ in \eqref{mottino} indeed equals $- \chi(\rho, \nu)$. It turns out that 
\be\label{eccolino}
-\k=\chi(\rho,\nu) =\rot{-}\bigl( {\l_c^+(\a)   C_{\rot{\nu}}^{\a+1}/\r}\bigr)^{\frac{1}{\a+1+d}}\,,
\en
where $\l_c^+(\a)\in (0,+\infty) $ is a constant of percolation type  depending  only from $\a$ and $d$ (although the dependence on $d$ does not appear in the notation). It has the following characterization.  Consider the \magenta{v.r.h.} resistor network on a PPP with density $\l$ and  i.i.d. energy  marks   in $[0,1]$ having  distribution $\nu_{1,\a}^+$  with  density proportional to $E^\a$ on $[0,1]$. Afterwards keep only filaments with conductivity at least $e^{-1}$ at inverse temperature $\b=1$. Then  the resulting network a.s. percolates for $\l>\l_c^+(\a)$ and a.s. does not percolate for $\l<\l_c^+(\a)$. 
\magenta{Once the preprint \cite{FH}   is validated, then, as stated in Corollary  \ref{immenso}, Mott's law \eqref{mottino} can be considered proved also for  $\nu\in \cV_\a$. In this case \eqref{eccolino} is valid with the constant $\l_c^+(\a)$ replaced by $\l_c(\a)$. The definition of $\l_c(\a)$ is the same of  $\l_c^+(\a)$  with exception that the energy marks have  value in $[-1,1]$  with  distribution $\nu_{1,\a}$ having  density proportional to $|E|^\a$ on $[-1,1]$. Finally, we point out that Corollary~\ref{immenso} comes from the asymptotically matching upper   and  lower bounds for the conductivity provided respectively by Theorem~\ref{parataUB} and Theorem~\ref{parataLB}.}

 \smallskip
 
 \rot{Experiments suggest that  Mott's law is universal, in the sense that it holds in a large class of disordered  media. One could imagine to model impurities and energies not by a marked PPP as above, which is also reasonable because impurities cannot be arbitrarily close (one could use e.g. a marked Delone process). In Section~\ref{MR_univ} we investigate  the emergence of Mott's law for more general marked SPPs, with i.i.d. marks.}
 
 $\bullet$ \rot{{\bf Theorem~\ref{paradiso}}}.
\rot{Also dealing  with other marked SPP,  PPPs play anyway a fundamental role. 
Indeed, 
one expects  that,}  at inverse temperature $\b$, just  points $x$  with marks  $E_x$ in a suitable window $[ -E(\b), E(\b)]$ give the main contribution to transport.
In Theorem \ref{paradiso} we  show the following. Let us  start with a generic 
ergodic stationary SPP with i.i.d. energy marks having distribution $\nu$  which asymptotically behaves as $\nu^+_{1,\a}$ or $\nu_{1,\a}$ while approaching  the origin. Then the rescaled set of marked points $\left( x/(\b E(\b)), E_x /E(\b)\right)$  with $|E_x|\leq  E(\b)$ 
converges to a  PPP with some intensity  $\l$ and with  i.i.d. energy marks 
having distribution respectively  $\nu^+_{1,\a}$ and $\nu_{1,\a}$ (the same random geometric structure entering in the definition of $\l_c^+(\a)$ and $\l_c(\a)$).

$\bullet$ \rot{{\bf Theorems~\ref{supersanto} and  \ref{santantonio}}}.
\rot{These theorems  give sufficient conditions to have lower and upper bounds in agreement with Mott's law for generic marked SPPs. In Section~\ref{universo} we motivate why these bounds should asymptotically match (for $\b\to+\infty$) for many marked SPPs with i.i.d. marks in view of Theorem~\ref{paradiso}.  We also derive  an expected  formula for  $\k$ (see \eqref{carezza}). We point out that Theorems~\ref{parataUB} and  Theorem~\ref{parataLB} for marked PPPs have been obtained by applying  Theorems \ref{supersanto} and  \ref{santantonio}, respectively.}\\

Our proof of the above  Mott's law relies also  on our previous results in  percolation \magenta{\cite{FH,FagMim1,FagMim2}} (obtained in collaboration \magenta{with I.~Hartarsky and with}  H.A.~Mimun) and stochastic  homogenization \cite{Fhom2}. Having expressed $\s(\b)$   in terms of the effective homogenized matrix in \cite{Fhom2}, we can use the variational characterization of the latter to obtain upper bounds on $\s(\b)$ via suitable test functions. On the other hand,  $\s(\b)$ \rot{is} the limit of the suitably rescaled conductivity of finite volume resistor networks as described in Section \ref{ssec_MA}. \rot{Hence} we can use Rayleigh's monotonicity law to lower bound 
$\s(\b)$ by thinning the finite volume resistor networks  \rot{keeping just some vertex-disjoint linear chains of electrical filaments with  enough conductance. Indeed, for the latter we can easily lower bound the effective conductivity by standard circuit theory. The procedure is  different} (and much refined) from the one in \cite{FM,FSS} and  now our upper and lower bounds on $\s(\b)$ asymptotically match as $\b\to+\infty$ (thanks also to suitable rescaling arguments). The test functions and thinned resistor networks used in \cite{FM,FSS}, respectively,  are of different nature, not optimal, while the ones presented here are inspired by the paradigm of Critical Path Analysis \cite{AHL} and require a detailed knowledge of a suitable percolation model investigated separately in   \magenta{\cite{FH,FagMim1,FagMim2}}.
The overall derivation of Mott's law  has required (and motivated)   the  results in \magenta{\cite{Fhom2,FH,FagMim1,FagMim2}} (although we treated there larger classes of models).

\smallskip

As next steps \rot{we would like to derive the scaling limit  \eqref{mottino1} for $d=1$ and investigate a larger class  of marked  SPPs (not given by PPPs) to which Theorems~\ref{supersanto} and~\ref{santantonio} can be applied to derive Mott's law}.

\smallskip

{\bf Outline of the paper}. In Section \ref{ssec_marked_SPP} we recall some basic notions of marked simple point processes on $\bbR^d$. In Section \ref{palmato} we introduce the Palm distribution 
and the effective homogenized matrix. In Section \ref{ssec_MA} we describe the \magenta{v.r.h.} random resistor network  and recall the scaling limit of its directional conductivity obtained in \cite{Fhom2}. In Section \ref{sec_mott_rw} we describe  Mott's random walk and recall some previous results  from \cite{CF1,Fhom1,Fhom2,FM,FSS}. In Section \ref{grafene} we  focus on the part of the \magenta{v.r.h} resistor network with conductances lower bounded by a given threshold and introduce the critical conductance threshold. Sections \ref{MR_PPP} and \ref{MR_univ} contains our main results. 
%
The other sections \magenta{and the appendixes} are devoted to the proofs.

We point out that several intermediate results are valid for all dimensions $d\geq 1$. We will state explicitly when we will restrict to $d\geq 2$.
%
\section{Some preliminaries on marked simple point processes on $\bbR^d$}\label{ssec_marked_SPP} 
In what follows, given a topological space $\cX$, we denote by $\cB(\cX)$ its Borel $\s$--algebra. 
 Given $A\in \cB(\bbR^d)$ we denote by $\ell(A)$ the Lebesgue measure of $A$. 
Given a measure $\mathfrak{m}$ and a function $f$ defined on the same space, we write $\mathfrak{m}[f]$ for the integral of $f$ w.r.t.~$\mathfrak{m}$.

 We denote by $\hat \O$ the space of  locally finite subsets of $\bbR^d$. Trivially, elements of $\hat \O$ are countable sets. Moreover, each $\hat \o \in \hat \O$ can be identified with the atomic measure $\sum_{x\in \hat \o }\d_x$, which  gives finite mass to bounded Borel subsets of $\bbR^d$. This allows to endow $\hat \O$ with the standard metric defined on the space of boundedly finite measures on $\bbR^d$ (cf. \cite[Eq.~(A2.6.1) in Appendix~A2.6]{DV}). The precise definition of this metric is not necessary for the rest and  we do not recall it here, but we will recall some of its properties when needed. We just mention here that a sequence $(\hat\o_n)_{n\geq 1}$  converges to $\hat \o$ in $\hat \O$ if and only if $\lim_{n\to +\infty}\hat\o_n[f]= \hat \o[f]$ for all   continuous  functions $f: \bbR^d\to \bbR$ with bounded  support (see  \cite{DV}[Thm.~A2.6.II]). 
 
 Due to  the above interpretation of $\hat \o$ as atomic measure, 
  given  $A\subset \bbR^d$,  $\hat \o (A)$ equals  the cardinality of $\hat \o \cap A$.
  One can prove that the Borel $\s$--algebra $\cB(\hat\O)$ is then generated by the sets $  \{\hat \o \in \hat \O\,:\, \hat \o (A) =n\} $ with $A\in \cB(\bbR^d)$ and $n \in \bbN$ (combine Prop.~7.1.III and  Cor.~7.1.VI in \cite{DV}).  From now on, $\hat \O$ will be thought of  as a \rot{set of measures} with $\s$--algebra of  measurable sets given by $\cB(\hat \O)$.

   Given $x\in\bbR^d$ and a  configuration $ \hat \o\in\hat \O$, we define the translated configuration $\t_x\hat \o$  as
\be\label{traslare}
\t_x\hat \o:= \{ x_i-x \}\; \text{ if }\; \hat \o=\{ x_i\}\,.
\en
\rosso{(above, and in the rest, we write $\{x_i\}$ for $\{x_i:i\in I\}$, omitting the index set).}
Each translation $\t_x:\hat\O\to \hat\O$ is a measurable map.
 A set $A\in \cB(\hat\O)$ is called \emph{translation invariant} if $A=\t_x A$ for all $x\in \bbR^d$.
 
 \begin{Definition}[SPP] \label{def_spp}
 A \emph{simple point process} (\rot{briefly} SPP) $\hat \xi$ on $\bbR^d$ is a random element of $\hat \O$. 
 $\hat \xi$ is called \emph{stationary} if its law $\hat \bbP$ satisfies $\hat \bbP( \t_x A)=\hat \bbP(A)$ for all $A\in \cB(\hat \O)$ and $x\in \bbR^d$. In this case, the \emph{intensity} of $\hat \xi$ is defined as $\rho:=\hat \bbE\bigl[ \hat\o([0,1]^d) \bigr]\in [0,+\infty]$, $\hat \bbE$ being the expectation w.r.t. $\hat \bbP$.
$\hat \xi$  is called \emph{ergodic} if it is stationary  and  $\hat \bbP(A)\in \{0,1\}$ for any translation invariant set $A\in \cB(\hat \O)$. 
\end{Definition}
For a stationary SPP with intensity $\rho$ it holds  $\hat \bbE\bigl[ \hat\o(A) \bigr]= \rho \ell(A)$ for any Borel set $A\in \cB (\bbR^d)$ \cite{DV}.

\medskip
In order to introduce the marked simple point processes,  
 we denote by $\O$ the space of subsets $\o=\{ (x_i, E_{x_i}) \}\subset\bbR^d\times \bbR$ such 
 that $\{x_i\}$ is a (countable) locally finite subset of $\bbR^d$ (i.e.~$\{x_i\}\in \hat \O$).
  The real number $E_{x_i}$ is called the \emph{energy mark} of the site $x_i$.
We will often  identify $\o=\{ (x_i, E_{x_i}) \} \in \O$ with the atomic measure $\sum_i \d_{(x_i,E_{x_i})}$. Note that, given  $A\subset \bbR^d\times \bbR$,   $\o(A)$  is the number of points of $\o$ in $A$.

If we denote by $\hat \O_{d+1}$ the metric space defined as $\hat \O$ above but with $\bbR^d$
replaced by $\bbR^{d+1}=\bbR^d\times\bbR$, one can prove that $\O$ is a Borel subset of $\hat \O_{d+1}$\footnote{$\O=\cap_{n=1}^\infty\magenta{(} \cup_{m=1}^\infty B_{n,m}\magenta{)}$ where $B_{n,m}$ is the set of 
 $\o \in \hat \O_{d+1}$ such   that $\o$ contains at most one point in the sets $(z,0) + [-1/m,1/m)^{d-1}\times \bbR$ for all $z\in [-n,n)^{d-1}\cap (\bbZ^{d-1}/m)$.}.
We consider $\O$ with the metric and topology induced from $\hat \O_{d+1}$. We  then get that  the Borel $\s$--field  $\cB(\O)$  is generated by the sets $\{ \o\in \O\,:\, \o(A)=n\}$, where    $A$ and $n$ vary respectively in $\cB( \bbR^d\times \bbR) $  and $\bbN$. Moreover,  a  sequence $(\o_n)_{n\geq 1}$ in $\O$ converges to $ \o\in \O$ if and only if $\lim_{n\to +\infty} \o_n[f]= \o[f]$ for all   continuous  functions $f: \bbR^d\times \bbR \to \bbR$ with bounded support.

%

    Given $x\in\bbR^d$ and  $ \o\in \O$, we define the translated configuration $\t_x\o$  as
\be\label{traslare_bis}
\t_x\o:= \{ (x_i-x, E_{x_i}) \}\; \text{ if }\; \o=\{ (x_i,E_{x_i})\}\,.
\en
 Finally, 
 we introduce the spatial projection
\be
\O \ni \o \mapsto \widehat{\o} \in \hat \O \text{ where } \widehat{\o}:=\{x_i\} \text{ if }
\o = \{(x_i, E_{x_i})\}\,.
\en
 In particular, we will frequently  write $\o=\{ (x, E_x)\,:\, x\in \widehat \o\}$. Both the translation  $\t_x:\O\to \O$ and the projection $\O\ni \o \mapsto \widehat{\o} \in \hat \O$ are measurable functions, as can be easily checked.

\begin{Definition}[Marked SPP] \label{def_mspp}
A \emph{marked simple point process} (\magenta{briefly}, marked SPP) $\xi$  on $\bbR^d$ is a random element of  $\O$.  $\xi$ is called \emph{stationary} if its law $\bbP$ satisfies $\bbP( \t_x A)= \bbP(A)$ for all $A\in \cB(\O)$ and $x\in\bbR^d$. In this case, the intensity of $\xi$ is defined as $\rho:=\bbE\bigl[ \widehat\o([0,1]^d) \bigr]\in [0,+\infty]$, $\bbE$ being the expectation w.r.t. $\bbP$.
$\xi$  is called \emph{ergodic} if it is stationary  and if $\bbP(A)\in \{0,1\}$ for any translation invariant set $A\in \cB(\O)$ \rosso{(i.e.~such that $\t_x A=A$ for all $x\in \bbR^d$).} 
 \end{Definition}

\begin{Warning} In general, the symbol $\; \hat{}$ will refer to SPPs. Space of configurations, generic configurations, SPPs, laws of SPPs will 
usually be denoted respectively by $\hat\O$, $\hat \o$, $\hat \xi$, $\hat \bbP$.
Note that the spatial projection of $\o\in \O$ is denoted by $\widehat{\o}$, while a generic element of $\hat \O$ is denoted by $\hat{\o}$. The context will also allow  to easily distinguish between   $\widehat{\o}$ and $\hat \o$.
\end{Warning}

The marked SPPs we are interested \rot{in} will be obtained  as
  $\nu$--randomization of a SPP \cite[\magenta{p.~225}]{Kal}. Let us recall  this procedure.

\begin{Definition}[$\nu$-randomization $\hat \bbP_\nu$] \label{def_randomization}
Consider   a SPP $\hat \xi$ on $\bbR^d$ with law  $\hat\bbP$ on $(\hat\O, \cB(\hat\O))$. Given a realization of $\hat \xi$, mark each point by i.i.d. random variables  (independently from all the rest) with common distribution $\nu$, and call $\xi$ the resulting marked SPP. 
We say that the law $\hat \bbP_\nu$ of 
  $\xi$ is the  $\nu$-randomization of $\hat \bbP$ and that   $\xi$ is the $\nu$--randomization of $\hat \xi$.  
 \end{Definition}
 
 We now recall the notion of $p$--thinning, which will be crucial in what follows:
 \begin{Definition}[$p$--thinning $\hat \bbP_p$] \label{def_thinning} Given a SPP  $\hat \xi$  on $\bbR^d$ with law $\hat \bbP$ and given $p\in [0,1]$, the $p$--thinning of $\hat \xi$ is the SPP  obtained  as follows. Given a realization $\hat\o$ of $\hat \xi$, for each point in $\hat \o$ one tosses  a biased coin with probability $p$ to give head \rot{(all coin tosses are independent)}: if the coin gives head the point is kept, otherwise it is deleted. We denote by $\hat\bbP_p$  the law of the $p$-thinning of $\hat\xi$. We will \magenta{briefly} say that $\hat\bbP_p$ is  the $p$-thinning of $\hat \bbP$.
 \end{Definition}

\subsection{Poisson point processes}
A special role in Mott's variable range hopping is played by Poisson point processes.

\begin{Definition} [\gino]
\label{def_PPP} We denote by ${\rm PPP}[\rho]$ the homogeneous Poisson point process  on $\bbR^d$ with intensity  $\rho$ and we denote by  ${\rm PPP}[\rho,\nu]$ the 
 $\nu$--randomization of ${\rm PPP}[\rho]$ (having automatically intensity $\rho$). We use the same symbols also for the associated laws, i.e.~${\rm PPP}[\rho](A)$ is the probability of the event $A\in \cB (\hat \O)$ for the homogeneous Poisson point process with intensity $\rho$, and similarly for  ${\rm PPP}[\rho,\nu](A)$ with $A\in \cB(\O)$.
 \end{Definition}
  We  recall that a SPP $\hat \xi$ on $\bbR^d$ is a {\rm PPP}$[\rho]$ whenever
  (i) for mutually disjoint Borel subsets $A_1,A_2, \dots, A_k$ in $ \bbR^d$, the random variables $\hat \xi(A_1),\hat \xi(A_2),\dots, \hat \xi(A_k)$ are independent;
  (ii) for any bounded   set  $A\in \cB(  \bbR^d)$, $\hat \xi(A)$ is a Poisson random variable with parameter $\r\, \ell(A)$.
In particular, it holds  $ E[ \xi(A)] =\hat\bbE[\hat \o(A)]=\rho \ell(A)$ for any $A \in \cB(\bbR^d)$.
 We  recall that {\rm PPP}$[\rho]$ and {\rm PPP}$[\rho,\nu]$ are stationary and ergodic \rosso{(see e.g.~\cite[Prop.~8.13]{LP}, whose proof can be generalized to the marked case)}.

The families of PPP's and marked PPP's are closed w.r.t. some fundamental operations, which play a central role in Mott's law:
\begin{Proposition}\label{cotechino}
The following holds:
\begin{itemize}
\item[(i)] The $p$--thinning of 
  ${\rm PPP}[\rho]$ is  given by 
${\rm PPP}[p \rho]$. 
\item[(ii)] By sampling $\hat\o\in \hat\O$ according to  ${\rm PPP}[\rho]$, the SPP obtained by rescaling each $x\in \hat \o $ into $x/\ell$ for some fixed $\ell>0$ has law ${\rm PPP}[\ell^d \rho]$.
\item[(iii)] Fixed $\g>0$, by sampling $\o\in \O$ according to  ${\rm PPP}[\rho,\nu]$,    the marked SPP $\o_\gamma:=\{(x,E_x)\,:\, |E_x|\leq \g \}$ has  law ${\rm PPP}[ p  \rho,\nu_\g]$, where $p:=\nu([-\g, \g])$ and $\nu_\g$ is the probability measure $\nu$ conditioned to the event $[-\g,\g]$.
    \end{itemize}
\end{Proposition}
\begin{proof} Properties (i) and (ii) are well known and follow easily from  the above characterization of ${\rm PPP}[\rho]$.
 Property (iii) follows from Lemma \ref{lemma_2021} in Section \ref{Tarvisio} and from Item (i).
\end{proof}
%
%
%
%
\section{Palm distribution and effective homogenized matrix $D(\b)$}\label{palmato}
In this section we introduce the effective homogenized matrix $D(\b)$. As discussed later,  it is  intimately  related to the conductivity properties of the \magenta{v.r.h.} random resistor network and to the diffusion properties of Mott's random walk. To define $D(\b)$ we  will  first need to introduce the Palm distribution.

In what follows, $\b$ is a positive number which, in the applications,  will correspond to the inverse temperature.
\begin{Definition}[Conductance field]
Given $\o\in \O$ we associate to each  \rosso{unordered} pair $\{x,y\}$ with $x,y\in \widehat{\o}$ and $x\not=y$ the \magenta{conductance}
\begin{equation}\label{conda}
c_{x,y}(\o,\b):=\exp\bigl\{ -  |x-y| -\b  ( |E_x|+ |E_y|+ |E_x-E_y|) \}\,.
\end{equation}
\end{Definition}
\rosso{Note that \eqref{conda} corresponds to \eqref{condu} by taking $\g:=2$ and by replacing $\b/2$ in \eqref{condu}  with $\b$. Trivially, by suitably rescaling distance and energy, the results obtained for \eqref{conda} can be extended to \eqref{condu}.}

Trivially, we have 
\be\label{balzetto}
|E_x|+ |E_y|+ |E_x-E_y| = \begin{cases} 
2 \max \big\{ |E_x|,|E_y|\big\} & \text{ if } E_x \cdot E_y \geq 0 \,,\\
2 |E_x-E_y| & \text{ if } E_x \cdot E_y \leq 0\,.
\end{cases}
\en

To simplify the notation below, it is convenient to set 
\[
c_{x,x}(\o,\b):=0 \qquad \forall \o \in \O\,,\; x\in \widehat{\o}\,.
\]

\begin{Definition}[Palm distribution]  \label{def_palm} Let $\bbP$ be the law of a stationary  marked SPP on $\bbR^d$ with finite and positive  intensity $\rho$. The Palm distribution $\vpz$ associated to $\bbP$  is defined  as  the unique probability measure on  $(\O,\cB(\O))$  concentrated on the Borel set $\O_0:=\{\o\in \O\,:\, 0 \in \widehat \o\}$ and given by 
\begin{equation}\label{marlena}
\vpz(A)= \frac{1}{\rot{\rho}} \int _{\O}d \bbP(\o)\int_{[0,1]^d}  d\widehat{\o} (x) \mathds{1}_A(\t_x \o)\,, \qquad \forall A \in \cB( \O) \,.
 \end{equation}
\end{Definition}
\rosso{Above, and in the rest, we denote by $\mathds{1}_A$ or $\mathds{1}(A)$ (according to better readability) the characteristic function of the event $A$.}

Equation \eqref{marlena} is a special case of Campbell's formula. A more general  Campbell's formula is the following (cf.~\cite[Eq.~(12.2.4)]{DV}, \cite[Thm.~1.2.8]{FKAS}): for any  nonnegative Borel function $f: \bbR^d\times \O \to[0,\infty) $   and for any function $f\in L^1( \bbR^d\times \O, dx \times \vpz)$ it holds 
 \begin{equation}\label{campanello}
 \int_{\bbR^d}dx  \int _{\O_0} d\vpz ( \o) f(x, \o) =\frac{1}{\rot{\rho}} \int _{\O}d \bbP(\o)\sum_{x  \in \widehat{\o}} f(x, \t_x \o) \,.
 \end{equation}
 
As discussed  in \cite[Chapter~12]{DV}, roughly $\vpz=\bbP(\cdot\,|\, 0\in \widehat \o)$.

When $\bbP=${\rm PPP}$[\rho,\nu]$, $\vpz$ can be realized as follows. Sample $\o\in \O$ with  law $\bbP$ and sample, independently from $\o$, the number    $E\in \bbR $ with law $\nu$. Then 
the configuration $\o \cup \{ (0,E)\}$ has law $\vpz$  (see \cite[Chapter~12]{DV}. \magenta{In} addition note that $\bbP( 0\in \widehat{\o})=0$, hence the probability that  in this construction  the origin appears with two different marks is zero). 

If $\bbP$ is the $\nu$--randomization of $\hat \bbP$ (i.e.~$\bbP=\hat\bbP_\nu$) and $\hat \bbP$ is the law of a stationary SPP with finite and positive  intensity $\rho$, then it is simple to check 
 that the Palm distribution $\vpz$ is the $\nu$--randomization of the Palm distribution $\hvpz$ associated to $\hat\bbP$ (i.e.~$\vpz=(\hvpz)_\nu$). \rosso{Indeed we recall that,}  similarly to \eqref{marlena}, $\hvpz$ is defined as the 
unique probability measure on  $(\hat\O,\cB(\hat \O))$  concentrated on the Borel set $\hat \O_0:=\{\hat \o\in \hat \O\,:\, 0 \in \hat \o\}$ and given by 
\begin{equation}\label{marlenabis}
\hvpz(A)= \frac{1}{\rot{\rho}} \int _{\hat\O}d \hat \bbP(\hat \o)\int_{[0,1]^d}  d\hat{\o} (x) \mathds{1}_A(\t_x \hat \o)\,, \qquad \forall A \in \cB( \hat \O) \,.
 \end{equation}

\begin{Definition}[Effective homogenized  matrix] \label{cuore} Let $\bbP$ be the law of a stationary  marked SPP on $\bbR^d$ with finite and positive  intensity $\rho$ and let $\vpz$ be the associated Palm distribution. Suppose that 
\be\label{2nd_moment}
\int_{\O_0} d\vpz(\o)\sum_{x\in \widehat \o}  c_{0,x}(\o,\b)|x|^2 <+\infty\,.
\en
Then
the \emph{effective homogenized  matrix}  $D(\b)$ is  
the unique $d\times d$ nonnegative symmetric matrix such that \rot{for all $a\in \bbR^d$ it holds}
 \begin{equation}\label{def_D}
 a \cdot D(\b) a =\inf _{ f\in L^\infty(\bbP_0) } \frac{1}{2}\int_{\O_0} d\vpz(\o)\sum_{x\in \widehat \o}  c_{0,x}(\o,\b) \left
 (a\cdot x - \nabla_x f (\o) 
\right)^2\,,
 \end{equation}
 where $\nabla_x f (\o) := f(\t_x \o) - f(\o)$.
\end{Definition}

Note that condition \eqref{2nd_moment} guarantees that the r.h.s. of \eqref{def_D} is finite (take $f\equiv 0$ there).

%
%
%
%
%
%
%
%
%
%
%

%
%
%
%
%
\section{\magenta{V.r.h.} random  resistor network}\label{ssec_MA}
In this section we recall some concepts and results from \cite{Fhom2} for what concerns the \magenta{v.r.h.} resistor network (we point out that  \cite{Fhom2} covers also many other resistor networks and uses \rot{ideas} of homogenization theory).

\smallskip

 In what follows $\b$ will denote the inverse temperature. 
\begin{Definition}[Resistor network \mamma$(\o)$] \label{def_MA}
Given  $\b>0$ and   $\o \in \O$, the  \magenta{v.r.h.}   resistor network \magenta{\rm RN}$[\b](\o)$  has node set $\widehat \o$ and has an electric  filament between any distinct nodes $x\not =y $ in $\widehat \o$ with \magenta{conductance} $c_{x,y}(\o,\b)$  \rot{defined in \eqref{conda}}.
\end{Definition}


%
%
%
%

\subsection{Finite volume  $\ell$--conductivity} \label{fumata}
Given $\ell>0$, we consider the stripe\footnote{We keep a 2-dimensional terminology, although $S_\ell$ lives in $\bbR^d$.}   $S_\ell:=\bbR\times (-\ell/2,\ell /2)^{d-1}$ and we decompose it as 
\be\label{colle1}
S_\ell=S_\ell^-\cup \L_\ell \cup S_\ell^+\,,\en
where 
\[S_\ell^-:=  \{x\in S_\ell\,:\, x_1\leq -\ell/2\}\,,\;  S_\ell^+:=  \{x\in S_\ell\,:\, x_1\geq \ell/2\}\,, \;
 \L_\ell:=(-\ell/2,\ell/2)^d\,.\]

\begin{figure}
\includegraphics[scale=0.20]{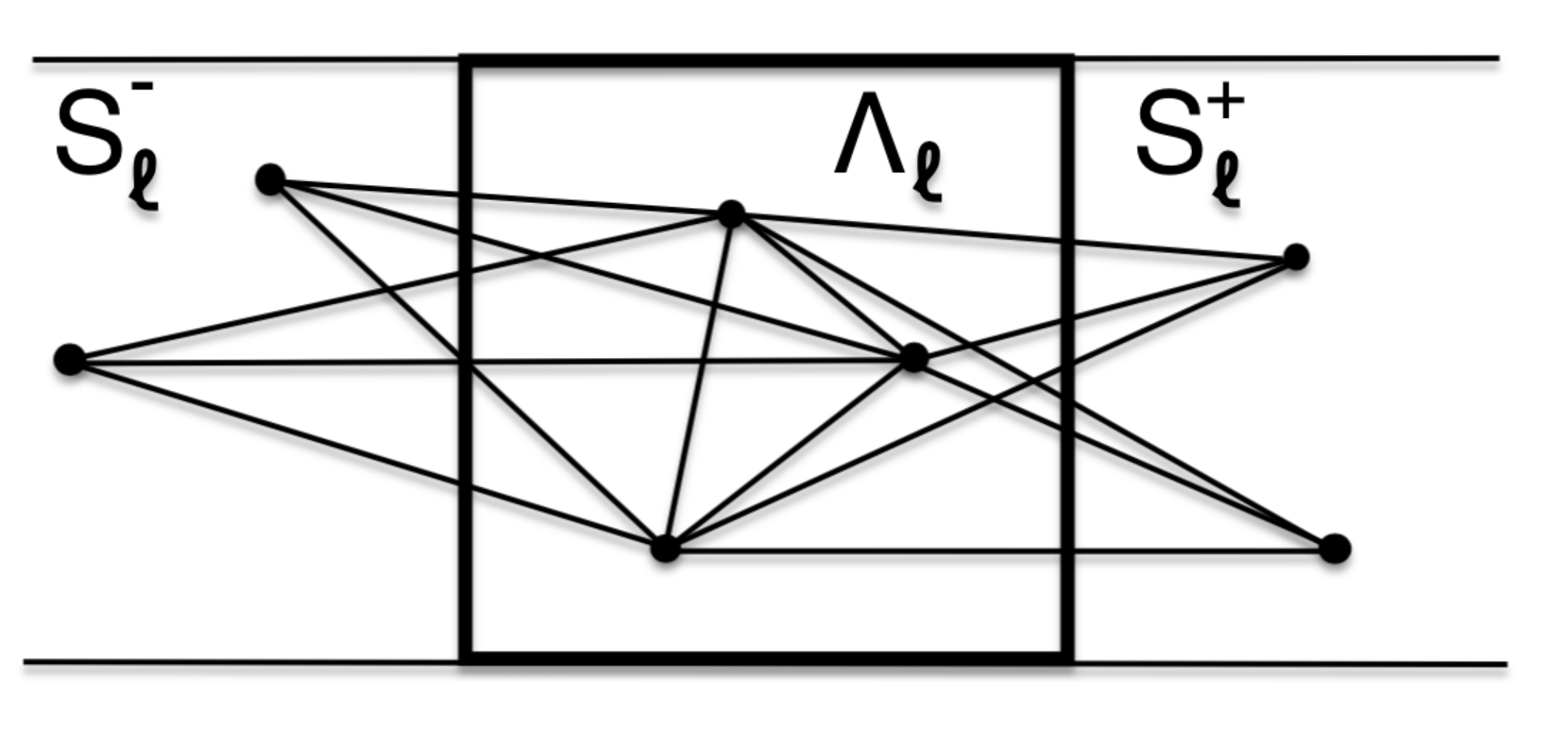}
\caption{A portion of the resistor network \mammaL. The box and the stripe correspond to  $\L_\ell$ and   $S_\ell$, respectively. }\label{messicano1}
\end{figure}

\rot{For the rest, 
we assume that 
\be\label{finito_totale}
\sum_{y\in \widehat\o }c_{x,y}(\o)<+\infty \qquad \forall x\in \widehat \o\,.
\en
Due to \cite[Lemma~4.1]{Fhom2} this holds $\bbP$--a.s. when 
 $\o$ is sampled with law $\bbP$ which is an ergodic stationary marked SPP with finite  and positive intensity satisfying  
\be\label{0_moment}
\int d\bbP^0(\o)\sum_{x\in \widehat \o}  c_{0,x}(\o,\b) <+\infty
\en
  (as in our applications).}


 \begin{Remark}\label{chitarrina} \rosso{As proved in \cite[Section~5.4]{Fhom1} conditions \eqref{2nd_moment} and \eqref{0_moment} are satisfied whenever $ \bbE[ \widehat{\o} ([0,1]^d)^2 ]  <+\infty$.} 
\end{Remark}
\begin{Definition}[Resistor network \mammaL$(\o)$] \label{def_MAelle}
Given  $\b,\ell>0$ and  $\o \in \O$,
 the v.r.h.  resistor network 
 \magenta{\rm RN}$[\b,\ell](\o)$ on the box $\L_\ell$ has node set $ \widehat \o \cap S_\ell$  and   is obtained by assigning to    each  pair  $\{x,y\}$ with   $x \in  \widehat \o \cap \L_\ell$ and $y \in \widehat \o \cap S_\ell$ 
 an electrical filament of \magenta{conductance} $c_{x,y}(\o,\b)$. 
\end{Definition}

We refer to  Figure \ref{messicano1} for a portion of \mammaL. 

\smallskip

\begin{Definition}[$\ell$--conductivity]  \label{def_ell_cond} Given  $\b,\ell>0$ and  $\o \in \O$,
we call $\s_\ell(\o,\b)$ the \emph{effective conductivity} of the resistor network \magenta{\rm RN}$[\b,\ell]$  along the first direction under a unit  electrical potential difference.  \rot{Namely},  $\s_\ell (\o,\b)$ is given by 
 \be\label{ide1}
\begin{split}
 \s_\ell(\o,\b)& :=  \sum _{x\in \widehat \o \cap S_\ell^-} \; \sum_{y \in  \widehat \o \cap \L_\ell}i^\o _{\b, \ell}(x,y) 
\end{split}
 \en
 \rot{where $ i^\o _{\b, \ell}(x,y) $   is the electric current flowing from $x$ to $y$.}
\end{Definition} 
\rot{We point out that condition \eqref{finito_totale} allows to  extend to  \mammaL$(\o)$ the standard Kirchoff's laws for electric circuits. Equivalently,  one can reduce the resistor network to a finite one as follows. Collapse all nodes in $\widehat \o \cap S_\ell^-$ (respectively, $\widehat\o\cap S_\ell^+$) into a single node with electric potential $1$ (respectively, $0$) and replace multiple filaments in parallel between any two nodes with a single filament whose conductance is the sum of the individual conductances, which  is finite due to   \eqref{finito_totale}). 
For later use
we  recall} 
a useful variational  characterization of $\s_\ell(\o, \b) $, usually called \emph{Dirichlet principle}  (cf.~\cite[Eq.~\rosso{(24)} and \rosso{Lemma~5.2}]{Fhom2} and \rosso{\cite[Exercise~1.3.11]{DS}}):
\be\label{papero}
\s_\ell (\o,\b)=\inf \bigl\{ \cD(u) \,\big{|} \, u: \widehat \o \cap S_\ell \to [0,1]\,, u_{|\widehat \o \cap S_\ell^-} \equiv 1
\,, \; u_{|\widehat \o \cap S_\ell^+}\equiv 0 \bigr\}\,,
\en
where $\cD(u)$ is the Dirichlet form
$ \cD(u) := \sum _{\{x,y\} \in \bbB^\o_\ell }c_{x,y}(\o,\b) \left( u(y)- u(x) \right)^2$.

\subsection{Infinite volume effective conductivity}
The following result is a special application of  \cite{Fhom2}[\rosso{Corollary~2.7}]\footnote{The hypotheses of Fact \ref{teo_eff_cond} assure the validity of Assumptions (A1),...,(A8) in \cite[Section~2]{Fhom2} (for (A7) and (A8) see \cite[Section~5.4]{Fhom1}), \rot{which  are at the basis of \cite{Fhom2}[Corollary~2.7]}.}.

\begin{Fact}[A. Faggionato  \cite{Fhom2}] \label{teo_eff_cond}
Let $\bbP$ be a stationary ergodic marked SPP  with finite and positive intensity $\rho$. Suppose in addition that 
\be\label{non_periodico}\bbP ( \o \in \O:  \t_x\o\not = \t_{x'} \o  \; \; \forall x\not =x' \text{ in }\rosso{\bbR^d})=1
\en
and that 
 $ \bbE[ \widehat{\o} ([0,1]^d)^2 ]  <+\infty$. 
If $e_1$ is an eigenvector of $D(\b)$, then 
\be\label{nanna}
\s(\b):=\lim _{\ell \to +\infty} \ell^{2-d} \s_\ell(\o,\b)=\rho  \rosso{D(\b)_{1,1}}\,.
\en
\end{Fact}

\rot{Let us give some  comments about  the assumptions of Fact~\ref{teo_eff_cond}}:

\begin{itemize}
\item
We recall that the bound $ \bbE[ \widehat{\o} ([0,1]^d)^2 ]  <+\infty$ imply both \eqref{2nd_moment} and \eqref{0_moment} (see Remark~\ref{chitarrina}).

\item
 As discussed in \cite[Section~2]{FSS}, condition 
\eqref{non_periodico}
 is equivalent to 
\be\label{non_periodico_bis}
\vpz ( \o \in \O:  \t_x\o\not = \t_{x'} \o  \; \; \forall x\not =x' \text{ in } \widehat{\o})=1\,.
\en
If $\bbP=\hat \bbP_{\nu}$ for some non-degenerate probability $\nu$, then \eqref{non_periodico} is automatically satisfied (cf.~\cite[Section~2]{FSS}).

\item
If $\bbP$ is reflection invariant  \rot{(namely, for each  $k=1,2,\dots, d$, $\bbP$ is left invariant when flipping the sign of the $k$--coordinate of all $x\in \widehat{\o}$)}
then $D(\b)$ is diagonal, while if $\bbP$ is isotropic then $D(\b)$ is a multiple of the identity  matrix (cf.~\cite{demasi}[Thm.~4.6--(iii)]). In the above two cases, trivially $e_1$ is an eigenvector of $D(\b)$.

\item 
If $\bbP={\rm PPP}[\rho,\nu]$, then  all the  assumptions of Fact \ref{teo_eff_cond}  are satisfied  and $D(\b)$ is a multiple of the identity matrix by isotropy.
\end{itemize}

\medskip

Due to Fact \ref{teo_eff_cond}, when  $e_1$ is an eigenvector of $D(\b)$, $\s(\b)=\rho\rosso{D(\b)_{1,1}}$ can be thought of as the effective infinite volume  rescaled conductivity along the first direction of the \magenta{v.r.h.} resistor network.

\rosso{Finally, although not used below,  we mention that \cite[Theorem~2.6]{Fhom2} gives the scaling limit of $\s_\ell(\o,\b)$ assuming more generally that $e_1\in {\rm Ker} (D(\b))$ or $e_1\in {\rm Ker}(D(\b))^\perp$.}

\section{Mott's random walk}\label{sec_mott_rw}
Given $\o \in \O$, Mott's random walk (introduced in \cite{FSS}) is the continuous-time 
Markov chain with state space $\widehat\o$ and with jump rates $c_{x,y}(\o,\b)$.  As random walk in a random environment, Mott's random walk  belongs to the class of random conductance models \cite{Bis}.
 
 To have a well defined Markov chain we assume that $c_{x}(\o,\b)<+\infty$ for all $x\in \widehat\o$ and that no explosion occurs. Due to \cite[Lemma~3.5]{Fhom1} and the discussion 
in \cite[Section~5.4]{Fhom1} devoted to Mott's random walk, the above  conditions are satisfied for $\bbP$--a.a.~$\o$ if $\bbP$ is a stationary ergodic  marked SPP with finite and positive intensity $\rho$ satisfying \eqref{non_periodico} and $\bbE[ \widehat{\o} ([0,1]^d)^2 ]  <+\infty$.

The effective homogenized matrix $D(\b)$ corresponds to the asymptotic diffusion matrix of Mott's random walk as formalized \rot{in} Fact~\ref{tolmezzo} below. We first introduce some notation.
 Given $\e>0$ we write $(P^\e _{\o ,t} )_{t \geq 0}$ for the $L^2(\mu^\e _\o)$--Markov semigroup associated to the random walk $(\e X^\o _{ \e^{-2} t} )_{t\geq 0}$ on $\e \widehat \o$, where $\mu^\e_\o := \e^d \sum _{x\in \widehat \o} \d_{\e x}$. Simply $P^\e _{\o ,t}f(\rot{\e x})$ is the expectation of $f ( \e X^\o _{ \e^{-2} t})$ when the random walk starts at $\e x$.
  Similarly we write $( P_t  )_{t \geq 0} $ for the  Markov semigroup on $L^2( \rho dx)$ associated to the  (possibly degenerate)  Brownian motion on $\bbR^d$  with diffusion matrix $2 D(\b)$. 
 The following fact is   a  very special application of a much more general result given by   \magenta{\cite[Theorem~4.3]{F_x_Francis}  (which relies on \cite[Theorem~4.4]{Fhom1})}:
 \begin{Fact}[A.~Faggionato \magenta{\cite{Fhom1,F_x_Francis}}] \label{tolmezzo}
 Let $\bbP$ be a stationary ergodic marked SPP  with finite and positive intensity $\rho$. Suppose in addition that $\bbP$ satisfies \eqref{non_periodico}
and that 
 $ \bbE[ \widehat{\o} ([0,1]^d)^2 ]  <+\infty$. 
\magenta{Then for $\bbP$--a.a.~$\o$ and for any $f\in C(\bbR^d)$ decaying at infinity fast enough (namely, such that for any $a>0$ there is $C>0$ with $|f(x)|\leq C(1+|x|)^{-a}$) the following holds.
 Given   $t\geq 0$, as $\e\da 0$ we have (i)  $ L^2(\mu^\e _\o ) \ni P^\e_{\o,t} f \to P_t f \in L^2(\rho dx)$, (ii) $\int d\mu^\e_\o(z) |  
P^\e_{\o,t} f (z) -P_t f (z) |^p=0$ for $p=1,2$.}
\end{Fact}


 The precise  meaning of the convergence $ P^\e_{\o,t} f \to P_t f$  is recalled in \cite[Section~3.5]{Fhom1}, \magenta{\cite[Section~4.1]{F_x_Francis}}.
 The above assumptions on the law $\bbP$ are the same of the ones in   Fact \ref{teo_eff_cond} (with exception that here we do not require that $e_1$ is an eigenvector of $D(\b)$). Hence the comments following  Fact \ref{teo_eff_cond} can be applied to the present setting \rosso{as well}.
 
 Stronger form of CLTs or even  invariance principles (annealed or quenched) have been proved under stronger conditions (we stress that the basin of application of Fact \ref{tolmezzo} is \rosso{very large}). 
 We recall that the annealed invariance principle of Mott's random walk has been obtained in \cite{FSS}. The quenched invariance principle has been obtained in \cite{CF1} for $d=1$ and in \cite{CFP} for $d\geq 2$. \rot{In \cite[Theorem~1.2--(2)]{CF1}   we have provided a criterion equivalent to the non-degeneracy of $D(\b)$ (e.g.~for ${\rm PPP}(\rho,\nu)$ $D(\b)>0$ if and only if $\rho> 1$) while in   \cite{CFP} the hypotheses} also  imply the non-degeneracy of  \rot{$D(\b)$} \magenta{(cf.~\cite{FH} for a larger class of marked SPPs with non-degeneracy $D(\b)$)}.
 
  Bounds on the diffusion matrix  in agreement with Mott's law have been obtained  under suitable assumptions in \cite[Thm.~1]{FM} and \cite[Thm.~1]{FSS}  for $d\geq 2$, while bounds 
 in agreement with an Arrhenius-type behavior  have been obtained in 
  \cite[Thm.~1.2\rot{--(2)}]{CF1} for $d=1$ \rot{when $D(\b)$ is non-degenerate and $\nu$ gives no mass to $0$}. 
 In particular \rosso{these} bounds are of the form 
  \begin{align*}
  & c_1 \exp \bigl\{ -c_1' \, \beta ^\frac{\a+1}{\a+d +1}\bigr\}\mathbb{I} \leq D(\b)\leq 
c_2 \exp \bigl\{ -c_2' \, \beta ^\frac{\a+1}{\a+d +1}\bigr\}\mathbb{I} \qquad &(d\geq 2)\,,\\
& c_1 \exp \bigl\{ -c_1' \, \beta  \bigr\} \leq D(\b)\leq  c_2 \exp \bigl\{ -c_2' \, \beta \}\qquad & (d=1)\,,
\end{align*}
for suitable $\b$--independent positive  constants $c_1,c_1',c_2,c_2'$.
\rosso{Under Fact \ref{teo_eff_cond}}, as pointed out in \cite[Corollary~3.1]{Fhom2},  the above  bounds imply similar bounds for  $\s(\b)/\rho$ due to \eqref{nanna}.
%
%
%
%
%

\section{Graph $\cG[\z,\b](\o)$ and  critical conductance for $d\geq 2$}\label{grafene}
\magenta{In this section we} restrict to dimensions  $d\geq 2$ (this will be understood in what follows).

\medskip

The infinite volume \magenta{v.r.h.}~resistor network, i.e.~\mamma$(\o)$ introduced in Definition \ref{def_MA}, corresponds to a weighted complete graph with vertex set $\widehat{\o}$. On the other hand, at small temperature, one expects that the  electrical filaments with a conductivity lower bounded by  a suitable $\b$--dependent threshold  give the main contribution to the electron transport, apart from a negligible term. It is therefore not surprising that a crucial role in our analysis will be played by the following  graph $\cG[\z,\b](\o)$:

\begin{Definition}[Graph \giggioest] \label{giggio_bello} Given $\o\in \O$, $\b>0$, $\z>0$, the  graph \giggio$(\o)$ has vertex set $\widehat \o$ and has edges $\{x,y\}$ with $x\not =y$ in $\widehat \o$ such that 
$c_{x,y}(\o,\b) \geq e^{-\z}$, i.e. 
\be \label{connectcont}
|x-y|+\beta(|E_x|+|E_y|+|E_x-E_y)|)\leq \z\,.
\en
\end{Definition}
Note that $\cG[\z,\b](\o)$  is the graph in $\bbR^d$ with vertex set $\widehat \o$ and supporting the family of filaments in the  resistor network \magenta{\rm RN}$[\b](\o)$ with conductivity at least $e^{-\z}$. \rot{Moreover, the graph $\cG[\z,\b](\o)$ is finite range,  since all its edges have length at most $\z$.}

\medskip

We say that the graph $\cG[\z,\b](\o) $ percolates if it has  some unbounded connected component and we set 
\be
\theta (\z,\b, \rho, \nu):= {\rm PPP}[\rho, \nu]  \Big( \cG [\z,\b] \text{ percolates} \Big)\,.
\en
We recall that, according to our notation, the above r.h.s. is the probability for ${\rm PPP}[\rho, \nu]  $ of the event that $\cG [\z,\b]$ percolates (see Definition \ref{def_PPP}). 

By ergodicity w.r.t.~spatial translations and since the existence of an unbounded connected component is a translationally  invariant event,  we have  that $\theta(\z,\b, \rho, \nu) \in \{0,1\}$. 
\rosso{As derived in Appendix \ref{guizzo}}, since $d\geq 2$  there exists 
a critical threshold  $\z_c\in (0, +\infty) $  depending on $\b$, $\rho$ and  $\nu$ such that 
\be\label{critico}
\begin{cases}
\theta(\z,\b, \rho, \nu) =1 & \text{ if } \z > \z_c(\b,\rho,\nu)\,,\\
\theta(\z,\b, \rho, \nu)=0 & \text{ if } \z< \z_c(\b,\rho,\nu)\,.
\end{cases}
\en

Similarly to \cite[Section~IV.B]{AHL} for the critical conductance, we set the following definition:
\begin{Definition}[Critical \rosso{threshold} $\z$ and conductance] \label{criticone}    The value $\z_c(\b,\rho,\nu)$  in \eqref{critico}  and  the value 
\[c_c(\b,\rho,\nu):= \exp\{-  \z_c(\b,\rho,\nu) \}\]
  are  called, respectively,  \emph{critical \rosso{threshold} $\z$}   and  \emph{critical conductance} of the  \magenta{v.r.h.}~random resistor network when $\bbP={\rm PPP}[\rho,\nu]$.
\end{Definition}

\section{Main results for  ${\rm PPP}[\rho, \nu]$, $d\geq 2$}\label{MR_PPP}

\rosso{In this section we restrict to dimension  $d\geq 2$ without further mention}.

\begin{Definition}[\rosso{Probabilities $\nu^+ _{\rot{C},\a}$ and $\nu_{\rot{C},\a}$}]\label{almograve}
Given   $\rot{C}>0$ and $\a\geq 0$, we introduce the probability measures $\nu^+ _{\rot{C},\a}$ and $\nu_{\rot{C},\a}$ with support $[0,\rot{C}]$ and $[-\rot{C,C}]$ respectively as follows:
 \begin{align}
&\nu^+ _{\rot{C},\a}(dE)=\frac{(\a+1)}{ \rot{ C} ^{\a+1} }  E^\a \rosso{\mathds{1}_{[0,\rot{C}]}(E)}  dE\,,\label{speciale2}\\
&\nu_{\rot{C},\a} (dE)=\frac{(\a+1)}{ 2 \rot{C} ^{\a+1} }  |E|^\a  \rosso{\mathds{1}_{[-\rot{C},\rot{C}]}(E)}   dE\,.\label{speciale1}
\end{align}
\end{Definition}
Recall the graph $\cG[\z,\b]$ introduced in Definition \ref{giggio_bello}.  
\begin{Definition}[Critical intensities $\l^+_c(\a)$ and $\l_c(\a)$]
\label{fiordo} Given $\a\geq 0$ we  define $\l^+_c(\a) \in(0,+\infty)$  as the critical intensity such that   the graph $\cG[1, 1]$ under ${\rm PPP}[\l,\nu^+_{1,\a}] $ a.s. percolates  if $\l >\l^+_c(\a)$, while  it 
a.s.   does not percolate if $\l<\l^+_c(\a)$.
We define  similarly  $\l _c(\a)\in(0,+\infty)$  by replacing  $\nu^+_{1,\a}$ with $\nu_{1,\a}$.
  \end{Definition}
\rosso{The existence and non-triviality of the above critical intensities  $\l^+_c(\a)$ and  $\l_c(\a)$ can be easily proved by stochastic domination (in particular,  by comparing the random graph $\cG[1,1]$ with suitable  Poisson Boolean models as in  \cite{FagMim1,FagMim2}).  We sketch the derivation in   Appendix \ref{guizzo}, where we also derive another useful characterization of the critical intensities in terms of the Palm distribution and the connected component of the origin}.


\begin{Definition}[Classes $\cvp$ and $\cv$, critical density $\l_c^*(\a)$] \rosso{Let $\a\geq 0$.}
 \label{monte_pisciano} 
$\;\; $
\begin{itemize}
\item We define $\cvp$ as the family of probability measures $\nu$ such that $\nu$ restricted to a neighborhood of the origin is proportional to $E^\a \mathds{1}(E\geq 0) dE$, i.e.~for some \rot{$\e_\nu,C_\nu>0$}   we have   $\nu(B)=\nu^+_{\rot{C_\nu},\a}(B) $ for each Borel set $B\subset[\rot{-\e_\nu,\e_\nu}]$.  \rot{If $\nu=\nu^+ _{C,\a}$, then we take   $\e_\nu:=C$ and $C_\nu:=C$}.

\item  We define $\cv$ as the family of  probability measures $\nu$ such that $\nu$ restricted to a   neighborhood of the origin is proportional to $|E|^\a  dE$,  i.e.~for some \rot{ $\e_\nu,C_\nu>0$}  we have   $\nu(B)=\nu_{\rot{C_\nu},\a}(B) $ for each Borel set $B\subset[-\rot{\e_\nu,\e_\nu}]$. \rot{If $\nu=\nu _{C,\a}$, then we take   $\e_\nu:=C$ and $C_\nu:=C$}.

\item When dealing with $\nu \in \cvp$ or $\nu \in \cv$ we set
\begin{equation}\label{nuziale}
\l_c^*(\a):=
\begin{cases}
\l^+_c(\a) & \text{ if } \nu\in \cvp ,\\
\l_c(\a) & \text{ if } \nu \in \cv\,.\\
\end{cases}
\end{equation}
\end{itemize}
\end{Definition}
\rot{We point out that  $\nu\in \cV_\a^+$ iff $\nu$ around the origin  has density exactly  of the form $c_\nu E^\a \mathds{1}(E\geq 0)$. In this case $c_\nu=(\a+1)C_\nu^{-\a-1}$. Similarly, $\nu\in \cV_\a$ iff $\nu$ around the origin  has density exactly  of the form $c_\nu |E|^\a$.  In this case, $c_\nu=2(\a+1)C_\nu^{-\a-1}$.}


Recall Definition \ref{criticone} and in particular that  $c_c(\b,\rho,\nu):= \exp\{-  \z_c(\b,\rho,\nu) \}$.
\begin{Theorem}[Critical \rosso{threshold} $\z$ and conductance   of $\magenta{{\rm RN}}\text{$[\b]$}$] \label{teo1}  \rosso{Assume that $\nu \in \cvp\cup\cv$}. Then 
 the critical \rosso{threshold} $\z$ and the critical conductance  of the resistor network
$\magenta{{\rm RN}}[\b](\o)$  with $\o$ sampled according to ${\rm PPP}[\rho, \nu]$  are given by 
\begin{align}
\z_c(\b,\rho, \nu )& = \bigl( {\l^*_c(\a)/\r}\bigr)^{\frac{1}{\a+1+d}}
  (  \b \rot{C_\nu}) ^{\frac{\a+1}{\a+1+d}}\,, \label{zetacritico}\\
c_c(\b,\rho, \nu )& = \exp\big\{- \bigl( {\l^*_c(\a)/\r}\bigr)^{\frac{1}{\a+1+d}}
  (  \b \rot{C_\nu}) ^{\frac{\a+1}{\a+1+d}}\big\}\, ,\label{ccritico}
\end{align}
if  $\beta$ is large enough. More precisely, \eqref{zetacritico} and \eqref{ccritico} hold  whenever the  r.h.s.  of \eqref{zetacritico}  \orc{is smaller than} $\min\{\rot{C_\nu, \e_\nu}\}\b$. 
\end{Theorem}

The proof of Theorem \ref{teo1} is provided in Section \ref{sec_proof_teo1}. 
Theorem \ref{teo1} immediately implies the following:
\begin{Corollary}[Value $\chi(\rho,\nu)$]\label{luna}
Suppose that $\nu \in \cvp\cup \cv$ and consider the resistor network
$\magenta{{\rm RN}}[\b](\o)$  with $\o$ sampled according to ${\rm PPP}[\rho, \nu]$. 
Then 
 \be\label{premium}
\chi(\rho, \nu):=\lim _{\b\to +\infty} \b^{- \frac{\rosso{\a+1}}{\rosso{\a+1+d}} } \ln c_c(\b,\rho,\nu )=  - \bigl( {\l^*_c(\a)   C_{\rot{\nu}}^{\a+1}/\r}\bigr)^{\frac{1}{\a+1+d}}
     \,.
\en
\end{Corollary}
\smallskip

We now present our results for Mott's law.  We \magenta{recall} that, since the PPP is isotropic, $D(\b)$ is a multiple of the identity and $e_1$ is always an eigenvector of $D(\b)$. In particular \eqref{nanna} is satisfied.

  \begin{Theorem}[Upper bound in Mott's law] \label{parataUB} Suppose that  $\nu\in \cvp\cup \cv$ and take $\bbP={\rm PPP}[\rho,\nu]$.
Then we have 
 \begin{align}
 &  \limsup _{\b\to +\infty} \b^{- \rosso{\frac{\a+1}{\a+1+d}} } \ln  D(\b)_{1,1}  \leq
  \chi (\rho, \nu)\,, \label{preUB}\\
&  \limsup _{\b\to +\infty} \b^{- \rosso{\frac{\a+1}{\a+1+d}} } \ln \s(\b)\leq   
\chi (\rho, \nu)
\,.\label{UB}
  \end{align}
\end{Theorem}
\verde{The proof of Theorem \ref{parataUB} is given in Section~\ref{sec_proof_parataUB} and is  based on   a more general result (Theorem \ref{supersanto}), which is  not restricted to PPPs.}
We point out that \eqref{UB} is an immediate consequence of \eqref{preUB} since we have  \eqref{nanna}. The strategy to prove \eqref{preUB} will be to use the variational formula \eqref{def_D} for $D(\b)$ for clever choices of the test function $f$ (partially inspired by the proof of \cite[Theorem~3.12]{PR} and several scaling considerations discussed in the next sections).

\smallskip

We now move to the lower bound in Mott's law.
To this aim, first we need some definitions.
Recall first the definition of $S_\ell, S_\ell^-, S_\ell^+, \L_\ell$ given at the beginning of Section \ref{fumata}.
\begin{Definition}\label{def_LR} Let $\bbG=(\bbV,\bbE)$ be a graph with $\bbV\subset \bbR^d$.
A left-right (LR) crossing  of  the box $\L_\ell$ in the graph  $\bbG$    is any  sequence of distinct points $x_1$, $x_2$,$\dots$, $x_n \in \bbV$  with $n\geq 3$  such that 
$\{x_i ,x_{i+1}\} \in  \bbE$ for all $i=1,2,\dots, n-1$;
$x_1 \in S_\ell^{-} $ and   $x_n \in S_\ell^+ $;
$x_2, x_3,\dots, x_{n-1}\in \L_\ell$.
We also define $\cN_\ell(\bbG)$ as the  maximal number of  vertex-disjoint LR crossings of $\L_\ell$ in $\bbG$.
\end{Definition}
\rosso{By the above notation $\bbG=(\bbV,\bbE)$, we mean that $\bbV$ is the vertex set and $\bbE$ is the edge set, with $\bbE\subset \big\{ \,\{x,y\} \,:\, x,y\in \bbV,\;x\not=y \text{ in }\bbV\big\}$.}

\begin{Definition}\label{good_LR_distr}
Let $\bbG=(\bbV,\bbE)$ be a random  graph  with $\bbV\subset \bbR^d$.
We say that $\bbG$ has a   \emph{good LR crossing  distribution} if   a.s. there exists 
 $c>0$ such that  $\cN_\ell(\bbG)\geq c\, \ell^{d-1}$ for $\ell $ large enough.
 \end{Definition}
We point out that the constant $c$ in the above definition can be random.

The following fact \rosso{(valid for $d\geq 2$)} is obtained by applying   \cite[Thm.~1]{FagMim2} and by using arguments similar to the ones in  \cite[Section~8]{FagMim2} to check the hypotheses. We detail its derivation in Appendix \ref{pinco}. We mention that \cite{FagMim2} covers a large class of graphs built on \magenta{PPPs}, including the   graph $\cG[1,1]$. 
\begin{Fact}[A.~Faggionato \& A.H.~Mimun \cite{FagMim2} and Appendix \ref{pinco}] \label{cortocircuito}
\rot{For any  fixed}
 $\l > \l^+_c(\a)$ there exist  constants $c_1,c_2>0$ depending on $\l$ \rosso{and} $\a$ such that 
 \be
{\rm PPP}[\l , \nu^+_{1,\a}] \left( \cN_\ell (\cG[1,1]) \geq  c_1 \ell^{d-1}\right)\geq 1- \exp\left \{- c_2 \,\ell ^{d-1}\right\}\,,
\en
for $\ell$ large enough.
In particular, by \rot{the} Borel-Cantelli lemma, the random graph $\cG[1,1](\o)$ with $\o$ sampled according to 
${\rm PPP}[\l , \nu^+_{1,\a}] $ has a good LR crossing distribution.
\end{Fact}

\begin{Remark} \magenta{The  results  in  \cite[Section~3]{FH} imply that Fact~\ref{cortocircuito} remains valid when replacing  $\l^+_c(\a), \nu^+_{1,\a}$ by $\l_{1,\a}, \nu_{1,\a}$, respectively. Since \cite{FH}  is not yet published at the moment, we mention this extension separately.} 
\end{Remark}
\rosso{The above fact on LR crossings is one of the main tools in the proof of the following result:}
  \begin{Theorem}[Lower bound in Mott's law] \label{parataLB} Suppose that  $\nu\in \cvp$ and take $\bbP={\rm PPP}[\rho,\nu]$.
Then we have 
 \begin{align}
 &  \liminf _{\b\to +\infty} \b^{- \rosso{\frac{\a+1}{\a+1+d}} } \ln D(\b)_{1,1}  \geq  \chi(\rho, \nu)\,, \label{preLB}\\
&  \liminf _{\b\to +\infty} \b^{- \rosso{\frac{\a+1}{\a+1+d}} } \ln \s(\b)\geq    \chi (\rho, \nu)
\,.\label{LB}
  \end{align}
  
  If in addition  the random graph $\cG[1,1](\o)$ with $\o$ sampled according to ${\rm PPP}[ \l, \nu_{1,\alpha}] $ has a good LR crossing distribution  for all $\l>\l_c(\a)$ \magenta{(as claimed in \cite{FH})}, then 
  \eqref{preLB} and \eqref{LB} hold also for  $\nu\in \cv$ and $\bbP={\rm PPP}[\rho,\nu]$.  \end{Theorem}
The proof of \verde{Theorem \ref{parataLB}} is provided in \verde{Section~\ref{sec_proof_parataLB} as a consequence of a more general result given by Theorem~\ref{santantonio}}. Due to \eqref{nanna},  \eqref{LB} implies \eqref{preLB}. We will  prove \eqref{LB} by means of  Rayleigh monotonicity law for resistor networks.

\medskip

Theorems \ref{parataUB} and \ref{parataLB} have the following immediate consequence:
\begin{Corollary}[Mott's law] \label{immenso}
Suppose that  $\nu\in \cvp$ and take $\bbP={\rm PPP}[\rho,\nu]$.
Then we have 
 \[  \lim _{\b\to +\infty} \b^{- \rosso{\frac{\a+1}{\a+1+d}} } \ln D(\b)_{1,1}  =     \lim _{\b\to +\infty} \b^{- \rosso{\frac{\a+1}{\a+1+d}} } \ln \s(\b)=    \chi (\rho, \nu)\,.
\]
  The same results hold for $\nu\in \cv$ and $\bbP={\rm PPP}[\rho,\nu]$ if 
the random graph $\cG[1,1](\o)$ with $\o$ sampled according to ${\rm PPP}[ \l, \nu_{1,\alpha}] $ has a good LR crossing distribution  for all  $\l>\l_c(\a)$  \magenta{(as claimed in \cite{FH})}.
\end{Corollary}

\begin{Remark} The property that the constant $\chi(\rho,\nu)$ appears in all the above asymptotics confirms the validity of the paradigm of critical path analysis (CPA) in our context (see \cite{Ji} for other CPA  rigorous results in   a different model of transport in disordered media).
\end{Remark}

\section{Main results concerning universality}\label{MR_univ}

 The results presented in this section hold for all dimensions $d\geq 1$, \magenta{but they become effective only for $d\geq 2$}.
 
 \rot{The derivation of Theorem~\ref{parataUB} and \ref{parataLB}}
  holds, \rot{to a large extent,} for generic marked SPPs. \rot{In this derivation (cf.~Sections \ref{sec_UB} and \ref{sec_LB})},
   a central role is played by a  mechanism of thinning and rescaling which ultimately  captures the backbone of  filaments in the \magenta{v.r.h.} resistor network mainly contributing to the electron transport. This mechanism is described by the marked SPP \eqref{gandhi} in Theorem \ref{paradiso} below, which will frequently appear in the next sections. 
  
 \blu{We show here  that, also  when starting with a marked SPP not of the form ${\rm PPP}[\rho,\nu]$, this mechanism asymptotically leads to ${\rm PPP}[\magenta{\l},\nu^+_{1,\a}]$ or ${\rm PPP}[\magenta{\l},\nu_{1,\a}]$ \magenta{for some density $\l$}, thus \magenta{supporting} the universality behind Mott's law as Physics law and the central role played by the  PPP  (see Section~\ref{universo}).}

 \begin{Definition}[Probability measure $\nu_{\star, \g}$] \label{stella_stellina}
Given   a probability measure $\nu$ on $\bbR$ having $0$ in its support and given $\g>0$, we define  $\nu_{\star, \g}$  as the probability measure with support inside $[-1,1]$ \rosso{such that}
 \begin{equation}
 \nu_{\star, \g}(B):= \frac{\nu(\g B \cap [-\g, \g])}{\nu([-\g, \g])}\,, \qquad B \in \cB (\bbR)\,.
 \end{equation}
 \end{Definition}
 Note that $ \nu _{\star,\g} ( \cdot )= P\big(\,  X/\g \in \cdot \,\big{|}\, |X/\g |\leq 1\,\big) $, where $X$ is a random variable with distribution  $\nu$.
\begin{Theorem}\label{paradiso} Consider a stationary ergodic SPP  on $\bbR^d$ with law  $\hat \bbP$  and  intensity $\rho \in (0,+\infty) $ and consider  a probability measure $\nu$ on $\bbR$ with $0$ in its support such that 
\begin{itemize}
\item[(i)] $\nu_{\star, \g}$  converges \rosso{weakly}  to $\nu^+_{1,\a}$ or to  $\nu_{1,\a}$ as $\g \da 0$;
\item[(ii)]    the following limit exists, finite and  positive:
 \be\label{vichinghi100}
C_*:=\lim _{\g \da 0} \frac{\nu([-\g, \g] )}{ \g ^{\a+1}}\in (0,+\infty)\,.
\en
\end{itemize}
Fix $\l>0$ and set 
\be\label{canditi105} \ell(\b):= \bigl(\l /(C_* \rho)\bigr)^{\frac{1}{\a+1+d}} \b ^{\frac{\a+1}{\a+1+d}}\,.
\en
Then, when  $\o $ is  sampled according to   $\hat{\bbP}_\nu$, the \verde{law $\bbP^\b$ of the} rescaled marked SPP \be\label{gandhi}
\magenta{\o_{\ell(\b),\b}:=}\left\{ \Big ( \frac{x}{\ell(\b)}, \frac{\b}{\ell(\b)} E_x\Big) \,: \, x\in\widehat \o  \,, |E_x| \leq \frac{\ell(\b)}{\b} \right \}
\en 
 converges weakly \rot{as $\b\to+\infty$} to   
\be\label{ghibaudo}
\begin{cases}
{\rm PPP}[\l, \nu^+_{1, \a}] & \text{ if } \lim_{\g\da 0}\nu_{\star, \g}=\nu^+_{1,\a}\,,\\
{\rm PPP}[\l, \nu_{1, \a}] & \text{ if }\lim_{\g \da 0}\nu_{\star, \g}=\nu_{1,\a}\,.
\end{cases}
\en
\end{Theorem}
The proof of the above theorem is given in Section \ref{sec_dim_paradiso}.   \verde{Theorem~\ref{paradiso} will not be used to derive other results. On the other hand, combined with the results below,  it helps to  justify  (although not rigorously yet) the universality of Mott's law (cf.~Section~\ref{universo}).}

\verde{The notation $\bbP^\b$ for the law of $\o_{\ell(\b),\b}$ (see \eqref{gandhi}) when $\o$ is sampled by $\bbP$ will appear also in the rest,  for $\ell(\b)$ proportional to $\b ^{\frac{\a+1}{\a+1+d}}$  as in \eqref{canditi105}.
 We stress that in our discussion $\b>0$, hence there is no notational conflict with the use of $\bbP^0$ for the Palm distribution associated to $\bbP$.}

\rosso{As stated in Proposition \ref{paradiso_PPP} below, if $\nu \in \cvp$ or $\nu\in\cv$, then conditions (i) and (ii) of the above theorem are automatically satisfied. Moreover, for} the marked SPPs treated in Section \ref{MR_PPP}, the limit point in the weak convergence of \eqref{gandhi} described above is achieved already for $\b$ large enough:


\begin{Proposition}\label{paradiso_PPP} Let $\nu\in \cvp\cup\cv$.
  Then, for    $0<\g \leq \min\{\rot{C_\nu,\e_\nu}\} $, it holds 
\begin{itemize}
\item[(i)]    $\nu_{\star, \g} = \nu^+_{1,\a}$ if $\nu \in \cvp$ and   $\nu _{\star, \g}=\nu_{1,\a}$ if $\nu\in \cv$; 
\item[(ii)] $\nu([-\g, \g])= \g^{\a+1}/C_{\rot{\nu}}^{\a+1}$, in particular   \eqref{vichinghi100} holds with $C_*:= C_{\rot{\nu}}^{-\a-1}$\,.
\end{itemize}
 Given $\rho, \l>0$ define  $\ell(\beta)$ as in \eqref{canditi105}  with  $C_*:= C_{\magenta{\nu}} ^{-\a-1}$, i.e.
\be\label{canditi200}
 \ell(\b):= \bigl(\l \magenta{C_{\nu}^{\a+1}} /\rho\bigr)^{\frac{1}{\a+1+d}} \b ^{\frac{\a+1}{\a+1+d}}\,.
\en
 Then, for $\beta$ large \rosso{(\magenta{specifically}, when $\ell(\b)/\b\leq \min\{\rot{C_\nu,\e_\nu}\}$)}, the law \verde{$\bbP^\b$}  of \eqref{gandhi} 
with $\o$ sampled according to ${\rm PPP}[\rho,  \nu]$ \verde{is given by}
\be\label{zoomo4}
\verde{\bbP^\b=
\begin{cases}
 {\rm PPP}[\l   ,\nu_{1,\a}^+]   &\text{ if } \nu\in \cV^+_\alpha\,,\\
{\rm PPP}[\l   ,\nu_{1,\a}]   & \text{ if } \nu \in \cV_\alpha\,.
 \end{cases}}
\en
\end{Proposition}

The proof of the above proposition  is given in Section \ref{Trento}.

\verde{In the next theorem, given $\o$ with $0\in\widehat{\o}$,  we call  $C[1,1](\o)$  the connected component  containing  the origin in the graph $\cG[1,1](\o)$.}
\begin{Theorem}[\verde{Upper bound in Mott's law for generic marked SPPs}] \label{supersanto}
\verde{Let $\bbP$ be  a stationary marked SPP with finite and positive intensity $\rho$  satisfying \eqref{non_periodico} and such that     $\bbP^0( E_0 \in (-s,s))>0$ for any $s>0$.
Given $\tilde C>0$ we set
\be\label{zoomo3}
\ell(\b):= (\l \tilde C ^{\a+1}/\rho)^{\frac{1}{\a+1+d}}\b ^{\frac{\a+1}{\a+1+d}}\,.
\en
We call 
  $\bbP^\b $ the law of the configuration $\o_{\ell(\b),\b}$ when $\o$ is sampled by $\bbP$, and we denote by $ \bbE^{\b}_0 $ the expectation w.r.t.~the Palm distribution associated to  $\bbP^\b $.
  We assume that there are 
  conjugate exponents $p,q\in (1,+\infty)$ such that   
    $ \widehat{\o}([0,1]^d) ^{ q +1}$ has finite expectation w.r.t. $\bbP$ and
    \be\label{gioia77}
     \limsup_{\b\to +\infty} {\bbE}^\b_0 \bigl[ ( {\rm diam}\, C[1,1])^{2p}  \big]<+\infty
    \en
 or more generally (instead of \eqref{gioia77})     \be\label{limitino77}
 \limsup_{\b\to +\infty} \b^{-\frac{\a+1}{\a+1+d}}  \ln {\bbE}^\b_0 \bigl[ ( {\rm diam}\, C[1,1])^{2p}  \big]\leq 0\,,
  \en
   Then   for any unit vector $v\in \bbR^d$ we have 
  \be\label{pingpong} \limsup _{\b\to +\infty}  \b^{- \frac{\a+1}{\a+1+d} } \ln (v\cdot D(\b) v) 
\leq
- (\l \tilde{C}^{\a+1} /\rho )^{\frac{1}{\a+1+d}} \,.\en}
\end{Theorem}
\verde{The proof of Theorem~\ref{supersanto} is given in Section~\ref{sec_UB}. We point out that, in Theorem~\ref{supersanto},  $D(\b)$ is well defined  due to our assumptions and Remark~\ref{chitarrina} (cf.~Definition~\ref{cuore}), while the Palm distribution associated to $\bbP^\b$ is well defined due to Lemma~\ref{quinoa} in the next section.}

\begin{Theorem}[\verde{Lower bound in Mott's law for generic  marked SPPs}] 
\label{santantonio} \verde{Consider a marked SPP $\bbP$ satisfying the same assumptions of Fact~\ref{teo_eff_cond}.
Given $\tilde C>0$ let $\ell(\b)$ be defined as in \eqref{zoomo3} and call
  $\bbP^\b $  the law of the configuration $\o_{\ell(\b),\b}$ when $\o$ is sampled by $\bbP$. Suppose that $\bbP^\b$--a.s. it holds 
\be\label{squid_game_33}
\limsup_{L\to +\infty} L^{1-d}\cN_L\left( \cG[1,1]\right)> 0 
\en
or, more generally (instead of \eqref{squid_game_33}), that for $\bbP$--a.a.~$\o$ it holds
\be\label{squid_game_3}
 \liminf_{\b \to +\infty}\limsup_{L\to +\infty}  \b^{-\frac{\a+1}{\a+1+d}} \ln \left( L^{1-d}\cN_L\left( \cG[1,1](\o_{\ell(\b),\b}\right)\right)\geq  0\,.
\en
Then
\be\label{polpetta} \liminf _{\b\to +\infty}  \b^{- \frac{\a+1}{\a+1+d} } \ln \s(\b)
\geq
- (\l \tilde{C}^{\a+1} /\rho )^{\frac{1}{\a+1+d}} \,.
\en}
\end{Theorem}
  \verde{The proof of Theorem~\ref{supersanto} is given in Section~\ref{sec_LB}.}
 
\subsection{Towards universality}\label{universo} \verde{Let $\hat \bbP$ and $\nu$ be as in Theorem~\ref{paradiso}. W.l.o.g. suppose that in Item (i) of Theorem~\ref{paradiso} the convergence is towards $\nu_{1,\a}$ (the other case is similar). We  now motivate why one should expect for many $\hat\bbP$ and $\nu$ as above that, given a unit vector $v\in \bbR^d$, 
\be\label{amore1}
 \limsup _{\b\to +\infty}  \b^{- \frac{\a+1}{\a+1+d} } \ln (v\cdot D(\b) v) 
\leq
- (\l^*_c(\a) /(C_*\rho) )^{\frac{1}{\a+1+d}} \,.
\en
Take $\ell(\b)$ as in \eqref{canditi105} with $\l<\l_c^*(\a)$. By Theorem~\ref{paradiso} the law $\bbP^\b$ of $\o_{\ell(\b),\b}$ with $\o$ sampled by $\bbP$ converges to ${\rm PPP}[\l,\nu_{1,\a}]$. The distribution of  ${\rm diam} \,C[1,1]$ has exponential tail when $\o$ is sampled according to the Palm distribution associated to ${\rm PPP}[\l,\nu_{1,\a}]$ (see~Fact~\ref{patatine} below from \cite{FagMim1}). This would suggest that for a large family of $\hat\bbP$ and $\nu$ as above, $ {\bbE}^\b_0 \bigl[ ( {\rm diam}\, C[1,1])^{2p}  \big]$ is well approximated by the analogous expectation w.r.t.~the Palm distribution of  ${\rm PPP}[\l,\nu_{1,\a}]$, which is finite. This would imply \eqref{gioia77}. To apply Theorem~\ref{supersanto} it would then been enough to impose the moment condition $\bbE\left[\widehat{\o}([0,1]^d) ^{ q +1}\right]<+\infty$ and one would get \eqref{pingpong}. By the arbitrariness of $\l<\l_c^*(\a)$ we then get \eqref{amore1}.}

\verde{By similar arguments one would expect from Theorem~\ref{paradiso} and Theorem~\ref{santantonio} (together with Fact~\ref{cortocircuito} from \cite{FagMim2} and \cite[Section~3]{FH})  that for a large class of marked SPPs as in Theorem \ref{paradiso}  it holds }
\be\verde{\label{amore2} \liminf _{\b\to +\infty}  \b^{- \frac{\a+1}{\a+1+d} } \ln \s(\b)
\geq
(\l_c^*(\a) /(\rho C_*) )^{\frac{1}{\a+1+d}} \,.}
\en

\verde{In conclusion, the above arguments would suggest that for a large class of marked SPPs as in Theorem \ref{paradiso}  Mott's law \eqref{mottino}  holds with constant $\k$ given by }
\be\label{carezza}
\verde{\k=- (\l_c^*(\a) /(\rho C_*) )^{\frac{1}{\a+1+d}} \,.}
\en

\section{Some preliminary results on marked SPPs}\label{prel_marked_SPP}

The results presented in this section hold for all dimensions $d\geq 1$.

\medskip

\begin{Definition}[Constant $\nu(\g)$ and probability measure $\nu_\g$] \label{san_benedetto}
Let  $\nu$ be   a  probability measure on $\bbR$ and let  $\g\geq 0$.  We set 
\begin{equation}\label{alloro}
\nu(\g):=  \nu \bigl([ -\g , \g]\bigr)\in [0,1]\,.
 \end{equation}
 If $0$ belongs to the support of $\nu$, given $\g>0$ we define the  probability measure   $\nu_{ \g} $ as
 \be
 \nu _{\g} ( B ) =\frac{\nu ( B \cap [-\g, \g])}{\nu([-\g, \g])} \,, \qquad B \in \cB (\bbR) \,.\label{sette1}
 \en
\end{Definition}
\verde{Recall Definition \ref{stella_stellina} of $\nu_{\star,\g}$.}
Note that the  probability measures  $\nu_{ \g} $ and  $\nu_{\star,\g}$ have  support contained in   $[ -\g , \g]$ and $[-1,1]$, respectively. Moreover, it holds 
\begin{align}
& \nu _{\g} ( B )= P\big(\,   X \in B \,\big{|}\, |X |\leq \g \,\big)\,,\label{sette1_bis}\\
&   \nu _{\star,\g} ( B )= P\big(\,  (X/\g) \in B \,\big{|}\, |X/\g |\leq 1\,\big) =\nu_\g(\g B) \,,\label{sette2}
 \end{align}
where $X$ is a random variable with distribution  $\nu$ and $B\in \cB(\bbR)$.

\begin{Definition}[$ \widehat\o_\g$, $\o_\g$] \label{porto_dascoli}
Given $\o \in \O$ and $\g >0$ we define 
\begin{align}
& \widehat\o_\g:=\{x\in \widehat \o\,:\, |E_x| \leq \g\}\,,\label{flautino1}\\
& \o_\g :=\{ (x, E_x)\,:\, x\in \widehat{\o} \,, \; |E_x| \leq \g\}\,. \label{flautino2}
\end{align}
\end{Definition}
Note that $\widehat{\,\o_\g\, } = {\widehat{\o}}_\g$. 

Recall Definitions \ref{def_randomization}  and \ref{def_thinning} of $\nu$--randomization and $p$--thinning, respectively. The above  \rosso{configurations}   $\widehat\o_\g$ and $\o_\g$ are obtained from  $\widehat\o$ and $\o$ by a kind of energy-based thinning.  We now clarify the relation between the above energy-based thinning and the random one introduced in Definition  \ref{def_thinning}. 

\begin{Lemma}\label{lemma_2021} Let $\g>0$,  let $\hat\bbP$ be  the law of a SPP  on $\bbR^d$   and let $\nu$ be a probability measure on $\bbR$.
If $\o$ is sampled according to $\hat \bbP_\nu$, then $\o_\g $ has law $(\hat \bbP_p)_{\nu_\g}$ where $p:=\nu(\g)\rosso{=\nu([-\g,\g])}$.\end{Lemma}
\begin{proof} 
By  \cite[Lemma~12.1]{Kal} we just need to prove that, 
given  any measurable function $f: \bbR^d\times\bbR\to [0,+\infty)$,  the random variable $\sum_{x\in \widehat{\o}_\g }f(x, E_x)$ with $\o$ sampled according to $\hat \bbP_\nu$ has the same law of the random variable $\sum_{x\in \widehat{\o} }f(x, E_x)$ with 
 $\o$ sampled according to $(\hat \bbP_p)_{\nu_\g}$.  \rosso{Moving to the Laplace transform}, it is enough to show that,  for any $f$ as above, it holds 
%
%
 \be\label{gente} \int_\O d \hat \bbP_\nu(\o)
 e^{- \sum_{x\in \widehat \o_\g} f(x,E_x)}  =\int _\O d (\hat \bbP_p)_{\nu_\g} (\o) e^{-\sum_{x\in \widehat \o} f(x,E_x)} \,.
 \en
\rosso{ Since $e^{- \sum_{x\in \widehat \o_\g} f(x,E_x)}=\prod_{x\in \widehat \o} e^{- f(x,E_x) \mathds{1}( |E_x| \leq \g )  }$}, it is easy to rewrite \eqref{gente}  as \begin{multline*}
\int _{\hat \O} d \hat \bbP(\hat \o) \prod_{x\in \hat \o} \Big[ \nu( [ -\g, \g]^c)+\int_{[-\g, \g]} \nu( dE)  e^{- f(x,E) }\Big]\\ = \int _{\hat \O} d \hat \bbP(\hat \o) \prod_{x\in \hat \o}     \Big[ 1-p+ p\int_{[-\g, \g]} \nu_\g ( dE)  e^{- f(x,E) }   \Big]  \,.
\end{multline*}
The claim follows from the identity  $p=\nu([-\g,\g])$ and the definition of $\nu_\g$.
\end{proof}

\verde{The following result will be applied to $\bbP=\hat\bbP_\nu$ in the next section, but the arguments used in its proof are completely general, hence we state it in full generality.}
\begin{Lemma}\label{quinoa} Let $\bbP$ 
be the law of a stationary marked SPP with   intensity $\rho\in (0,+\infty)$ and let 
 $\vpz$ be the Palm distribution associated to    $\bbP$.  Fix $\g>0$ and 
    denote by $\bbQ$ the law of $\o_\g$ when $\o$ is sampled according to $\bbP$. 
Assume that   $\bbP( \o_\g \not =\emptyset)>0$.
Then $\bbQ $ is a stationary marked SPP  and \orc{its intensity, denoted by $\rho_\g$, is finite and positive}. Moreover, 
$\vpz(0\in \rosso{\widehat\o}_\g)=\rho_\g/\rho $ and 
 the Palm distribution $\vqz$ associated  to  $\bbQ $
equals 
  the law of $\o_\g$,  with $\o$ sampled according to $\vpz(\cdot | 0\in \rosso{\widehat\o}_\g)$.
\end{Lemma}
\begin{proof}  \rosso{We recall  our notation \eqref{flautino1}:  $\widehat\o_\g:=\{x\in \widehat \o\,:\, |E_x| \leq \g\}$
(note that $\widehat{\o_\g}=\widehat{\o}_\g$).}
The stationary of $\bbQ $ is trivial. The intensity $\rho_\g$  of $\bbQ $  is finite since  \rot{it is} upper bounded by the intensity  $\rho$ of $\bbP$. If  $\rho_\g $ was zero, then by stationarity we would have $\int_\O d\bbP(\o)  \widehat{\o}_\g (A)=0 $ for any $A\in \cB(\bbR^d)$. This would imply   that $\bbP( \o_\g =\emptyset)=1$, thus contradicting our assumption.

It remains to  prove the last statement in the lemma concerning Palm distributions.
To this aim let $h : \O_0 \to \bbR$ be a nonnegative measurable function. 
We  use Campbell's identity \eqref{campanello}  with  $\bbP$ replaced by $\bbQ$ and with test function $f(x,\o):= \mathds{1}_{[0,1]^d}(x) h(\o)$. We then get 
\be\label{sonnone1}
\begin{split}
\int _{\O_0} d \vqz (\o) h(\o) & = \frac{1}{\rho_\g} \int _\O d \bbQ (\o) \sum_{x\in \widehat{\o}\cap [0,1]^d} h(\t_x \o)\\
& = \frac{1}{\rho_\g} \int _\O d \bbP(\o) \sum_{x\in \widehat{\o}_\g \cap [0,1]^d}  h\bigl(\t_x (\o_\g)\bigr) \\
&= \frac{1}{\rho_\g} \int _\O d \bbP(\o) \sum_{x\in \widehat{\o}\cap [0,1]^d}h\bigl(\t_x (\o_\g) \bigr)\mathds{1}(|E_x|\leq \g) \,.
\end{split}
\en

We now apply Campbell's identity   \rosso{\eqref{campanello}  with} $f(x,\o)=  \mathds{1}_{[0,1]^d}(x) h(\o_\g) \mathds{1} ( 0\in \widehat{\o}_\g)$. Using that $(\t_x \o)_\g=\t_x (\o_\g)$, we get 
\begin{equation}\label{sonnone2}
   \int _{\O_0} d\vpz ( \o) h(\o_\g) \mathds{1}(0\in \widehat{\o}_\g )  =\frac{1}{\rho} \int _{\O}d \bbP(\o)  \sum_{x\in \widehat{\o}\cap [0,1]^d}   h\bigl(\t_x(\o_\g)\bigr) \mathds{1}(|E_x| \leq \g) \,.
 \end{equation}
 In the special case $h\equiv 1$  in \eqref{sonnone2} we get
 \begin{equation}\label{sonnone3}
   \vpz(0\in \widehat{\o}_\g ) 
   =  \frac{1}{\rho} \int _{\O}d \bbP(\o)  \sum_{x\in \widehat{\o}\cap [0,1]^d}   \mathds{1}(|E_x| \leq \g)=
   \frac{\bbE\bigl[ \widehat{\o}_\g([0,1]^d)\big]}{\rho}
 = \frac{\rho_\g}{\rho}\,. 
 \end{equation}
 By combining \eqref{sonnone1}, \eqref{sonnone2} and \eqref{sonnone3} we get 
 \be
 \int _{\O_0} d \vqz (\o) h(\o)=\frac{1}{ \vpz (0\in \widehat \o_\g ) }  \int _{\O_0} d\vpz ( \o) h(\o_\g) \mathds{1}(0\in \widehat \o_\g ) \,,
 \en
 thus proving our final claim in Lemma \ref{quinoa}.
\end{proof}

\verde{Rescaling arguments  will reveal crucial. To this aim we introduce the following:}
\begin{Definition}[Homothety $\Psi_z$ \verde{and law $\hat \bbP\circ \Psi_z ^{-1}$}] \label{def_zaino}
Given $z >0$ we define $\Psi_z: \bbR^d\to \bbR^d$ as the map $\Psi_z (x) := x/z$. \verde{If we have} a SPP $\hat{\xi}$ on $\bbR^d$ with law $\hat\bbP$,  we write $\hat \bbP\circ \Psi_z ^{-1}$  for the law of $\psi_z(\hat{\xi})=\{x/z\,:\, x\in \hat\xi\}$, \magenta{i.e.}
$\bigl(\hat \bbP\circ \Psi_z ^{-1}\bigr)(B):= \hat \bbP\bigl( \Psi_z^{-1} (B)\bigr) $ for all $B \in \cB (\hat\O)$.
\end{Definition}

 Recall \eqref{alloro},  \eqref{sette1}, \eqref{sette2} and \eqref{flautino2}.
 The following technical fact will be frequently used in the rest:
 \begin{Lemma}
\label{fadeev_250}  Let  $\z,\b,\rho >0$, let $\nu$ be a  probability measure   on $\bbR$ and let $\hat\bbP$ be a  SPP.
   Sample 
 $\o$  according to   $\hat \bbP_\nu$ and    set  $p:=\nu( \z/\b)$. 
 Then
$\o_{\z/\b}$ has law $(\hat \bbP_p)_{\nu_{\z/\b}} $ 
and 
  \be\label{box_pokemon}
  \o_{\z,\b}:=\left\{ \big (x/\z, \b E_x/\z\big) \,: \,  x\in \widehat{\o}_ {\z/\b} \right \}
\en
  has law  $(\hat \bbP_p\circ \Psi_{\z}^{-1})_{\nu_{\star,  \z/\b}}$.  In particular, if    $\hat \bbP={\rm PPP}[\rho]$ and $\hat \bbP_\nu={\rm PPP}[\rho,\nu]$, then 
 \[ 
  (\hat \bbP_p)_{\nu_{\z/\b}} ={\rm PPP}[\rho \nu(\z/\b) ,\nu_{\z/\b}]\qquad 
   (\hat \bbP_p\circ \Psi_{\z}^{-1})_{\nu_{\star,  \z/\b}}= {\rm PPP}[\rho \nu(\z/\b)\z^d   ,\nu_{\star,  \z/\b}] \,.
  \]
\end{Lemma}
Note that $\o_{\z,\b}$ equals \eqref{gandhi} when $\z=\ell(\b)$.
\begin{proof}
  By sampling $\o$ according to $\hat\bbP_\nu$, $\o_{\z/\b}$ has law $(\hat \bbP_p)_{\nu_{\z/\b}} $  by    Lemma \ref{lemma_2021}.
The fact that $\o_{\z,\b}$ has law $(\hat \bbP_p\circ \Psi_{\z}^{-1})_{\nu_{\star,  \z/\b}}$ follows  in particular from the observation that  if   the random variable  $E$ has law  $\nu$, then conditioning to the event  $|E|\leq \z/\b$ the random variable 
$\b E /\z $ has law $\nu_{\star,  \z/\b}$ (see  \eqref{sette1_bis} and  \eqref{sette2}).
The conclusion concerning marked PPP's follows from 
 Proposition \ref{cotechino}. 
\end{proof}

\subsection{Convergence of marked SPPs}
\verde{The results in this subsection are used only in the proof of Theorem~\ref{paradiso}.}
\begin{Lemma}\label{weak_conv}
Consider a sequence of SPPs on $\bbR^d$  with laws $\pn$ and a sequence of probability measures $\mu_n$ on $\bbR$. Suppose that there exists a SPP   with law $\hat\bbP$ and a  probability measure $\mu$ on $\bbR$ such that  
$ \pn \,\Rightarrow \, \hat \bbP$  and $\mu_n \,\Rightarrow \, \mu$ as $n\to +\infty$.
Suppose also that $\int_{\hat \O}  d \hat \bbP(\hat \o ) \hat \o (\partial A)=0$ for any box $A\subset \bbR^d$.
Then $
 ( \pn)_{\mu_n} \,\Rightarrow \, \hat \bbP_\mu$ as $n\to +\infty$.
\end{Lemma}

Above  the symbol $\Rightarrow$ refers to weak convergence of probability measures \rosso{\cite{Bi}. When $\mu_n= \mu$ for all $ n$, the above result follows from \cite[Exercise~15,\,Chp.~16]{Kal}.}

 \begin{Remark}\label{condimento}The condition $\int_{\hat \O}  d \hat \bbP(\hat \o ) \hat \o (\partial A)=0$  for boxes $A\subset \bbR^d$  is satisfied in particular by any stationary \magenta{$\hat\bbP$} with finite intensity, as the Lebesgue measure gives zero mass to $(d-1)$--dimensional sets. 
 \end{Remark}

\begin{proof} Below we identify the configurations $\hat\o\in \hat\O$ and $\o\in \O$ with the atomic measures $\sum_{x\in \hat\o} \d_x$ and $\sum_{x\in \widehat \o} \d_{(x,E_x)}$ respectively, keeping the same names $\hat\o$ and $ \o$. In particular, $\o[f]=\sum_{x\in \widehat{\o} }f(x,E_x)$. Moreover, we denote by $f\in C^+_c(\bbR^{d+1})$ the family of 
  continuous functions $f\geq 0$ on $\bbR^{d+1}$ with compact support.

 A marked SPP on $\bbR^d$ is in particular a SPP  on $\bbR^{d+1}$, thus allowing to apply \cite{Kal}[Thm.~16.16]. 
 As a byproduct  with the properties of Laplace transform,
  to prove that $(\pn)_{\mu_n} \,\Rightarrow \,  \hat \bbP_\mu$ it is enough to show \rosso{for any $f\in C_c^+(\bbR^{d+1})$} that 
 \be\label{prati} \int d(\pn)_{\mu_n} (\o) e^{- \o [f]} \to  \int d\hat \bbP _{\mu} (\o) e^{- \o [f]}  \;\; \text{ as }n\to+\infty\,.
 \en

 \rosso{Let 
$\cG$ be the class of functions  
$g$ of the 
 form $g=\sum_{i=1}^N c_i \mathds{1}_{A_i \times B_i}$, where all  $c_i$'s are positive  numbers, 
 all $A_i\times B_i$ are pairwise disjoint, 
  all $A_i $  are boxes of the form $\prod_{k=1}^d [a_k,b_k)$ and all $B_i$ are intervals of $\bbR$ whose extremes have zero $\mu$-measure. For any $\e>0$ one can find $f^{\pm }_\e\in \cG$  with support included in the neighborhood of the support of $f$ with radius $1$, such that $f^-_\e\leq f \leq f^+_\e$ and  $\|f- f^\pm _\e\|_\infty \to 0$ as $\e\da 0$. 
  Hence, by dominated convergence, $\int d \hat \bbP_{\mu} (\o) e^{- \o [f^\pm_\e]}\to \int d \hat \bbP_{\mu} (\o) e^{- \o [f]}$ as $\e \da 0$, and then one 
  easily \magenta{derives}  \eqref{prati} with    $f\in C_c^+(\bbR^{d+1}) $  once getting it for any
  $f\in \cG$.
 }

Let us now   prove \eqref{prati}
for any given $f=\sum_{i=1}^N c_i \mathds{1}_{A_i \times B_i}$ in $\cG$. At cost to split each  $A_i$ into smaller boxes, 
  we assume that $A_1,A_2, \dots, A_k$ are pairwise disjoint, while $A_{k+1},\dots, A_N$ equal some sets between $A_1,A_2, \dots, A_k$. Hence 
  we can rearrange the sum to write $f$ in  the form 
$
 f(x,E)=\sum_{i=1}^k \mathds{1}_{A_i}(x)  g_i(E)$, where $ g_i (E) := \sum _{\substack{j: 1\leq j \leq N\\
 A_j=A_i}} c_j \mathds{1}_{B_j} (E)$.
Then  we have 
\be\label{sole}
\begin{split}
 \int d(\pn)_{\mu_n} (\o) e^{- \o [f]} 
& =
 \int d \pn (\hat \o)\exp\Big\{ \sum_{i=1}^k \hat \o (A_i)  \ln \mu_n[ e^{-g_i}] \Big\}\,.
\end{split}
\en
For each $i=1,\dots, k$, \rosso{$g_i:\bbR\to \bbR$ is a piecewise-constant function, with a finite family of   discontinuity points having zero $\mu$-mass. Hence   $\mu_n [ e^{-g_i}]\to \mu[e^{-g_i}]$  as $n\to \infty$.} 
Hence, fixed $\d>0$, for $n$ large enough we have
\begin{multline}\label{fulmini}
 \int d \pn (\hat \o)\exp\Big\{ \sum_{i=1}^k \hat \o (A_i)  (\ln \mu [ e^{-g_i}] -\d) \Big\}
\leq 
\int d(\pn)_{\mu_n} (\o) e^{- \o [f]} 
\\ \leq 
 \int d \pn (\hat \o)\exp\Big\{ \sum_{i=1}^k \hat \o (A_i)  \{ \ln \mu[ e^{-g_i}] +\d \}_- \Big\}\,.
\end{multline}
For any  box $A_i$ it holds  $\hat \o (\partial A_i)=0$ for $\hat \bbP$--a.a.~$\hat \o$ as 
  $\int d \hat \bbP(\hat \o ) \hat \o (\partial A_i)=0$ by assumption.  Since in addition the sets $A_i$ are relatively compact, 
 we can apply \cite[Thm.~16.16]{Kal} getting that  the vector $(\hat \o (A_1), \hat \o (A_2), \dots , \hat \o (A_k))$ with $\hat \o$ sampled according to $\pn$  converges in distribution to the same vector with $\hat \o$  sampled now according to $\hat\bbP$. \rosso{This implies the convergence of expectations of bounded continuous functions applied to the random  vectors. Hence we get}
  \begin{align}
 & \lim_{n\to +\infty} \text{l.h.s. of \eqref{fulmini}}= \int d \hat\bbP (\hat \o)\exp\big\{ \sum_{i=1}^k \hat \o (A_i)  (\ln \mu [ e^{-g_i}] -\d) \big\}\,, \label{merenda1}\\
  & \lim_{n\to +\infty} \text{r.h.s. of \eqref{fulmini}}= \int d \hat\bbP (\hat \o)\exp\big\{ \sum_{i=1}^k \hat \o (A_i)  \{ \ln \mu[ e^{-g_i}] +\d \}_- \big\}\,.\label{merenda2}
\end{align}
By dominated convergence,  the limits as $\d\da 0$ of \eqref{merenda1} and \eqref{merenda2} both equal
\[ \int d \hat\bbP (\hat \o)\exp\big\{ \sum_{i=1}^k \hat \o (A_i)   \ln \mu[ e^{-g_i}]  \big\}=\int d\hat \bbP_{\mu} (\o) e^{- \o [f]} \,.
\]
 By the above observations, taking the limit $n\to +\infty$ and afterwards the limit  $\d \da 0$ in \eqref{fulmini}, we get \eqref{prati}.
  \end{proof}

\verde{By applying Lemma~\ref{weak_conv} we get the following result:}
\begin{Lemma}\label{la_bomba}
Fix $c>0$ and, for each $p\in (0,1)$, fix a number $z(p)>0$ and  a probability measure $\mu^{(p)}$ on $\bbR^d$. Suppose that 
(i)  $\lim_{p\da 0} p\, z(p)^d=c$, 
(ii) $\mu^{(p)}\Rightarrow  \mu$ as $p \downarrow 0$.
Then, for any  stationary  ergodic   SPP  on $\bbR^d$ with law $\hat \bbP$ and finite  intensity $\rho$,  it holds 
$
\bigl( \hat \bbP_p \circ \Psi_{z(p)}  ^{-1}\bigr)_{\mu^{(p)}} \Longrightarrow {\rm PPP}[c \rho, \mu]$ as $p\downarrow 0$.
\end{Lemma}
\begin{proof}[Proof of Lemma \ref{la_bomba}]
Since {\rm PPP}$[c \rho,\mu]$ is the $\mu$--randomization of {\rm PPP}$[c\rho]$\rosso{,} since 
$\mu^{(p)}\Rightarrow  \mu$ \rosso{and due to Remark \ref{condimento}, we can apply   Lemma \ref{weak_conv}. As a consequence,} it is enough to prove that  $ \hat \bbP_p \circ \Psi_{z(p)}  ^{-1}\Rightarrow {\rm PPP}[c \rho]$.
 Let us call $\hat \xi$ the SPP  with law $\hat\bbP$ (shortly, $\hat \xi \stackrel{\cL}{\sim} \hat \bbP$).
 Then, we define $\rosso{\hat \xi}^{(p)}$ as the SPP  $\rosso{\hat \xi}^{(p)}:= \Psi_{z(p) } (\hat \xi)$.  Note that $\rosso{\hat \xi}^{(p)} \stackrel{\cL}{\sim}  \hat \bbP\circ \Psi_{z(p)}  ^{-1}$. We finally define $\rosso{\hat \z}^{(p)}$ as the $p$--thinning of $\rosso{\hat \xi}^{(p)}$. \rosso{Then}
$\rosso{\hat  \z}^{(p)} \stackrel{\cL}{\sim}  (\hat \bbP\circ \Psi_{z(p)}  ^{-1})_p=\hat \bbP_p\circ \Psi_{z(p)}  ^{-1}$. \rosso{Hence, we need to prove that the SPP ${\hat\z}^{(p)}$  converges \rosso{weakly} to  {\rm PPP}$[c \rho]$.}
 As  {\rm PPP}$[c \rho]$ corresponds to a Cox process directed by $c  \rho dx$ (cf.~\cite[Chp.~12]{Kal}), by  \cite[Thm.~16.19]{Kal} it is enough to prove that the random measure  $p\rosso{\hat  \xi^{(p)} }$ (defined as $[p\rosso{\hat \xi}^{(p)} ](B):=p \cdot \rosso{\hat \xi}^{(p)}(B)$)  converges \rosso{weakly}  to the deterministic measure $c \rho dx$.  Due to \cite[Thm.~16.16]{Kal}  it is enough to prove that, \rot{for a fixed} continuous function $f:\bbR^d \to \bbR$ with compact support, the random variable \rosso{$p \sum _{x \in \rosso{\hat \xi}^{(p)} } f(x)$}  
  converges \rosso{weakly} to $c \rho \int f(x) dx$ as $p\downarrow 0$. Setting $\e=\e(p):= 1/z(p)$, we have 
 \[p \sum _{x \in \rosso{\hat \xi}^{(p)} } f(x) =
 p \sum_{x\in \hat \xi}  f( x/z(p) )=( p /\e(p)^{-d})  \e(p)^d   \sum_{x\in \hat \xi}  f(\e(p) x) 
\]
 Note that $\e (p) \da 0$ as $p\da 0$ since  $p z(p) ^d\to c>0$ as $p\da 0$. 
 As $\hat \xi$ is an ergodic stationary SPP,   $\e^d  \sum_{x\in \hat \xi} f(\e x)$ converges a.s. (and therefore weakly) to $\rho \int f(x) dx$  as $\e\da 0$ (see e.g.~\cite[Prop.~3.1]{Fhom1},  \cite{T}). Since  $p /\e(p)^{-d}=p z(p) ^d\to c$ as $p\downarrow 0$,  $p \sum _{x \in \rosso{\hat \xi}^{(p)} } f(x) $  converges  \rosso{weakly} to $c \rho \int f(x) dx$ as $p\downarrow 0$.
 \end{proof}



\section{Rescaling of  $\cG[\z,\b]$ and proof of Theorem \ref{teo1}}\label{escher}

The results presented in the first part of this section hold for all dimensions $d\geq 1$.  In Subsection  \ref{sec_proof_teo1} we will restrict to $d\geq 2$.
 
 \smallskip
 
In this section we investigate the effect of rescaling on the random graph $\cG[\z,\b]$ introduced in  Definition \ref{giggio_bello}. We recall that, given $\o\in \O$, 
$\cG[\z,\b](\o)$  has vertex set $\widehat{\o}$  and edges given by pairs $\{x,y\}$ with $x\not =y$ satisfying 
\eqref{connectcont}, i.e.~$|x-y|+\beta(|E_x|+|E_y|+|E_x-E_y)|)\leq \z$. Afterwards we will apply these results to derive  Theorem \ref{teo1}.

\smallskip

For the next \orc{lemma} recall Definition \ref{stella_stellina}  of  $\nu_{\star,\g}$, Definition   \ref{san_benedetto} of  $\nu(\g)$ and  $\nu_\g$ and  
recall Definition \ref{porto_dascoli} of  $\widehat{\o}_\g$ and $\o_\g$.

\begin{Lemma}
\label{fadeev_new}  Let  $\z,\b,\rho >0$  and let $\nu$ be a  probability measure   on $\bbR$.  Then the  following holds:
\begin{itemize}
\item[(i)] For any $\o\in \O$
the graph $\cG[\z, \b](\o)$  is  the  union of the graph
$\cG[\z,\b](\o_{\z/\b})$   and the  singletons $\{x\}$ with  $x\in \widehat \o  \setminus \widehat \o_{\z/\b}$, which are isolated points in $\cG[\z, \b](\o)$. 
\item[(ii)]   
Calling $\Psi_\z \cG[\z,\b] (\o_{\z/\b})   $  the image of $ \cG[\z,\b] (\o_{\z/\b}) $ under the map  $\Psi_\z:x\mapsto x/\z$, the graph $\Psi_\z \cG[\z,\b] (\o_{\z/\b}) $  equals the graph $\cG[1,1](\o_{\z,\b})$ where  
  $\o_{\z,\b}:=\left\{ \big (x/\z, \b E_x/\z\big) \,: \,  x\in \widehat{\o}_ {\z/\b} \right \}
  $
  as in \eqref{box_pokemon}.
 \end{itemize}
\end{Lemma}
\begin{proof}
We start with Item (i).
Due to \eqref{connectcont}, if $ |E_x|>\z/\b$, then $x$ is an isolated point in the graph $\cG[\z,\b](\o) $. The conclusion is then immediate. Item
(ii) follows from the observation that \eqref{connectcont} is equivalent to 
$
\big|x/\z -y/\z\big |+\left(\big|\b E_x/\z \big|+\big|\b E_y/\z\big|+\big| \b E_x/\z-\b E_y/\z\big|\right)\leq 1.$
\end{proof}

An immediate consequence of   \orc{Lemmas  \ref{fadeev_250} and \ref{fadeev_new}} is the following fact:
\begin{Corollary}\label{funghetto} We have
\[
{\rm PPP}[\rho,
\nu]\big(\, \cG[\z, \b] \text{ percolates}\,\big)
= {\rm PPP}[\rho \nu(\z/\b)\z^d   ,\nu_{\star,  \z/\b}]
 \big( \, \cG[1,1]\text{ percolates}\, \big)\,.
 \]
\end{Corollary}
For $\nu$ given by  \eqref{speciale2} or \eqref{speciale1} or in general for $\nu\in \cvp \cup  \cv$,  the above r.h.s. can be made more explicit due to \magenta{Lemma~\ref{fissato} below. We do not prove it since it} is an immediate consequence of Items (i) and (ii) in Proposition \ref{paradiso_PPP} applied with $\g= \z/\b$. \verde{Proposition \ref{paradiso_PPP} is proved in the next section}. 

\begin{Lemma}\label{fissato} \rosso{Let  $\a\geq 0$, $\nu \in \cvp\cup\cv$ and  let $\z,\b>0$ be  such that $\z\leq \min \{\rot{C_\nu, \e_\nu}\}\b $. Then $\nu_{\star,\z/\b} $ equals $\nu^+_{1,\a}$ when $\nu\in \cvp$ and equals $\nu_{1,\a}$ when $\nu\in \cv$. Moreover, $  \rho \nu(\z/\b)\z^d    = \rho  \left(  \b \rot{C_\nu}  \right) ^{-(\a+1)} \z^{\a+1+d}$.}
\end{Lemma}


\subsection{Proof of Theorem \ref{teo1}}\label{sec_proof_teo1} We take $d\geq 2$. 
We prove the theorem  for $\nu\in \cvp$. The   case $\nu\in \cv$  is similar.
To simplify the notation we call $\eta$ the r.h.s. of \eqref{zetacritico}.
We need to show  for  $\eta<  \min\{\rot{C_\nu , \e_\nu}\}  \b $  that 
\begin{itemize}
\item[(1)]  if $\z<\eta$ then ${\rm PPP}[\rho,\nu] \left(\, \cG[\z, \b] \text{ percolates}\, \right)=0$, 
\item[(2)] if $\z>\eta$ then ${\rm PPP}[\rho,\nu]\left(\, \cG[\z, \b]\text{ percolates}\, \right)=1$.  
\end{itemize}
In Item (1)  we have \orc{$\z <  \min\{\rot{C_\nu , \e_\nu}\} \beta$} as $\eta <   \min\{\rot{C_\nu , \e_\nu}\} \beta$ by assumption. 
In Item (2) we can   restrict to $\z\in \left(\eta, \min\{\rot{C_\nu , \e_\nu}\} \beta\right]$ as this  set  is nonempty and the function  $\z\to  {\rm PPP}[\rho,\nu]\bigl( \cG[\z, \b] \text{ percolates}\, \bigr)$ does not decrease. Hence  in both Items (1) and (2)  we can restrict  $\z \leq \min\{\rot{C_\nu , \e_\nu}\} \beta $. Then, as a byproduct of   Corollary \ref{funghetto} and Lemma \ref{fissato}, we have 
  \begin{multline*}
  {\rm PPP} [\rho,\nu] \bigl( \cG[\z, \b] \text{ percolates}\, \bigr)
  = 
 {\rm  PPP}[\rho \nu(\z/\b) \z^d ,\nu_{\star,  \z/\b}] \bigl(  \cG[1, 1]  \text{ percolates}\,\bigr )\\
  ={\rm PPP}[ \rho  \left(  \b \rot{C_\nu}  \right) ^{-(\a+1)} \z^{\a+1+d}, \nu^+_{1,\a}]\bigl(  \cG[1, 1]  \text{ percolates}\,\bigr )\,.
 \end{multline*}
 We now observe that $\z\lessgtr \eta$ if and only if 
\[  \rho  \left(  \b \rot{C_\nu}  \right) ^{-(\a+1)} \z^{\a+1+d} 
\lessgtr  \rho  \left(  \b \rot{C_\nu}  \right) ^{-(\a+1)} \eta^{\a+1+d}=\l^+_c(\a)
\]
(the last identity follows immediately from \eqref{zetacritico}). 
  As $\l^+_c(\a)$ is the critical density for the graph  $\cG[1, 1](\o)$ with $\o$ sampled according to  ${\rm PPP}[\l,\nu^+_{1,\a}] $, we can conclude.

\section{Proof of Theorem \ref{paradiso} and Proposition \ref{paradiso_PPP}}\label{Tarvisio}
In this section we prove Theorem \ref{paradiso} and Proposition \ref{paradiso_PPP}.
\subsection{Proof of Theorem  \ref{paradiso}}\label{sec_dim_paradiso} We set
\be \label{meta}
p:=\nu( \ell(\b)/\b)\,, \qquad \g := \ell(\b)/\b\,, \qquad z = \ell (\b)\,.
\en
By Lemma \ref{fadeev_250}, \rot{the marked SPP defined in}
  \eqref{gandhi} has law $(\hat \bbP_p \circ \Psi_z^{-1} ) _{\nu_{\star,\g}}$.

We now want to apply Lemma \ref{la_bomba}. 
Firstly we   note that, by \eqref{vichinghi100} and \eqref{canditi105},   for $\b \to +\infty$ \rot{it holds}
\[ 
p z^d= \nu\bigl( \ell(\b)/\b\bigr) \ell(\b)^d
\rot{=(1+o(1)) }C_*  \bigl(\ell(\b)/\b) ^{\a+1} \ell(\b)^d 
\rot{=(1+o(1))}\l/ \rho \,.
\]
As a consequence $\lim_{\b \to +\infty }p z^d = \l/\rho$.
  Secondly  we observe that 
  \[
  \g:=\ell(\b)/\b= \bigl(\l /(C_* \rho)\bigr)^{\frac{1}{\a+1+d}} \b ^{-\frac{d}{\a+1+d}}
  \] goes to zero when $\b \to+\infty$. \rosso{Then,} by \eqref{vichinghi100},  
  $p:=\nu( \ell(\b)/\b)=\nu(\g)$  goes to zero as  $\b \to +\infty$. \rosso{Moreover,} \rosso{by Item (i) of Theorem \ref{paradiso}}, $\nu_{\star,\g}$ converges to $\nu^+_{1,\a}$ or $\nu_{1,\a}$ as $\b \to +\infty$. \rosso{Due to the above observations}, the weak convergence of \eqref{gandhi} to \eqref{ghibaudo}  follows from Lemma \ref{la_bomba}.

\subsection{Proof of  Proposition \ref{paradiso_PPP}}\label{Trento}
%

We prove Items (i) and (ii) only for $\nu \in \cvp$ as the case $\nu \in \cv$ is similar.
 We first prove our claim for $\nu=\nu_{\rot{C},\a}^+$. In this case  $\rot{\e_\nu=C_\nu=C}$, hence the upper bound  $\g \leq \rot{\min\{C_\nu, \e_\nu\}}$ reads $\g\leq \rot{C}$.
 Let $X$ be a random variable with law $\nu$. For $0 \leq t \leq \rot{C}$ we have $\nu(t)=\rosso{\nu([-t,t])=}(t/\rot{C})^{\a+1}$. Then, 
given $0\leq u \leq 1$ and  using that  $\g \leq \rot{C}$, we get 
 \begin{equation*}
 \nu_{\star,\g} ([0,u])= P\bigl(  X/\g\in [0,u]\,\big{|}\, X/\g \leq 1\bigr) =  \nu ( u\g)/\nu(\g)   =  u^{\a+1}=\nu^+_{1,\a} (  [0,u])\,.
  \end{equation*}
  This proves  that $  \nu_{\star,\g} =\nu^+_{1,\a}$.  Trivially, $\nu([-\g, \g])= \nu(\g)=(\g/\rot{C})^{\a+1}$.   Items (i) and (ii) are  then proved for $\nu=  \nu_{\rot{C},\a}^+$.
 We now move to a generic $\nu\in \cvp$. The value $\nu(\g)$ and the probability  $\nu_{\star,  \g}$ depend only on the measure $\nu$ restricted to $[-\g,\g]$  which equals the measure
 $\nu_{\rot{C_\nu},\a}^+$ restricted to $[-\g,\g]$ (as $\g\leq\rot{ \e_\nu}$). Hence  
    we have  $\nu(\g)=\nu'(\g)$  and  $\nu_{\star,  \g}=\nu'_{\star,  \g}$ where $\nu':=\nu_{\rot{C_\nu},\a}^+$.  The conclusion then follows from the special case 
 $\nu=\nu_{\rot{C},\a}^+$ treated above since $\g\leq \rot{C_\nu}$.

\smallskip
We move to the proof of  the final statement.
If one starts with $\o$ sampled by  ${\rm PPP}[\rho, \nu]$, then by Lemma \ref{fadeev_250}
%
%
%
 the law of \eqref{gandhi} is $
{\rm PPP}\left[\,\rho \nu( \g ) \ell(\b)^d  , \nu_{*,\g} \,\right]$, \magenta{where $\g:=\ell(\b)/\b$}.
As soon as  $\b$ is  large enough to assure that  $\g\leq \min\{\rot{C_\nu,\e_\nu}\}$,  by Item (ii) we have 
\[
\rho \,\nu( \g ) \ell(\b)^d = \frac{\rho \g^{\a+1}}{\rot{C_\nu^{\a+1}}} \ell(\b)^d= \frac{\rho}{\rot{C_\nu^{\a+1}}} \ell(\b)^{\a+1+d} \b^{-\a-1}= 
\l\,.\]
As a consequence, $
{\rm PPP}\left[\,\rho \nu( \g ) \ell(\b)^d  , \nu_{*,\g} \,\right]=
{\rm PPP}\left[\,\l , \nu_{*,\g} \,\right]$. Item (i) allows now to conclude.

\section{\verde{Proof of Theorems~\ref{parataUB} and~\ref{supersanto} (upper bounds in Mott's law)}}\label{sec_UB}

\verde{In this section we prove Theorem~\ref{supersanto}. By applying it, in Section \ref{sec_proof_parataUB} we prove Theorem~\ref{parataUB}}. 

\verde{To simplify the notation we take $v:=e_1=(1,0,\dots,0)$ but the arguments below are completely general as can be easily checked. Then to prove Theorem~\ref{supersanto} we}  need to upper bound 
 \begin{equation}\label{poh}
 D(\b)_{1,1} =\inf _{ f\in L^\infty(\bbP_0) } \frac{1}{2}\bbE_0 \Big[
 \sum_{x\in \widehat \o}  c_{0,x}(\o,\b) \left
 (x_1 - \nabla_x f (\o) 
\right)^2\Big]\,,
 \end{equation}
 where $\nabla_x f (\o) := f(\t_x \o) - f(\o)$.  \orc{Above, $\bbE_0$ denotes the expectation w.r.t. $\vpz$}.
 
 Due to the above variational characterization \eqref{poh} of $D(\b)_{1,1}$, for each $f\in L^\infty (\vpz)$ one has an upper bound of $D(\b)_{1,1}$. Our construction of the right functions $f$, giving efficient upper bounds, is inspired by the one in \cite{PR} (see Item (3) in
  \cite[Proof~of~Thm.~3.12]{PR}).
 
\medskip

Given a subset 
$A\subseteq\bbR^d$, we write  $\text{diam}\,A$ for the diameter of $A$ w.r.t. the uniform norm, i.e. $\text{diam}\,A:=\sup\{|x_i-y_i|\,:\,x,y\in A,\,i=1,\ldots,d\}$.

\begin{Definition}
 Given $x\in\bbR^d$, we denote by $C[\z,\b](x,\o)$   the connected component  containing  $x$ in the graph $\cG[\z,\b](\o)$.  If $x=0$, we simply  write  $C[\z,\b](\o)$.
 \end{Definition}

Given $n \in \bbN$ and given the parameters $\z,\b>0$, we consider the test function  $f_n=f_{n,\z,\b}\in L^\infty(\O_0)$ defined as \begin{equation}
\label{ftest}
f_n (\o):=
\begin{cases}
-\min \left\{ x_1:  x\in C[\z,\b](\o)\right\}  &\text{if } {\rm diam}\,C[\z,\b](\o) \leq n \,,
\\0 &\text{otherwise}\,.
\end{cases}\end{equation}
By \eqref{poh}, we have
\be\label{st1} 
2 D(\b)_{1,1} \leq \bbE_0 \Big[ \sum_{x\in \widehat \o}  c_{0,x}(\o,\b) \left
 (x_1 - \nabla_x f_n  (\o)\right)^2\Big]
 \leq A_1+A_2+A_3 \,,
\en
where
\begin{align}
& A_1:= \bbE_0 \Big[ \sum_{x\in \widehat \o}  c_{0,x}(\o,\b) \left
  (x_1 - \nabla _x f_n  (\o)\right)^2  \mathds{1} _{H_x}(\o)
  \mathds{1}_{
  \{ {\rm diam}\, C[\z,\b] \leq n \}} (\o)
  \Big]\,,\\
  & A_2:= \bbE_0 \Big[ \sum_{x\in \widehat \o}  c_{0,x}(\o,\b) \left
  (x_1 - \nabla _x f_n  (\o)\right)^2  \mathds{1} _{H_x}(\o)
  \mathds{1}_{
  \{ {\rm diam}\, C[\z,\b] > n \}}(\o)
 \Big]\,,\\
& A_3= \bbE_0 \Big[ \sum_{x\in \widehat \o}  c_{0,x}(\o,\b) \left
 (x_1 - \nabla_x f_n  (\o)\right)^2  \mathds{1} _{H^c_x}(\o) \Big]\,,
\end{align}
and $H_x:=\{\o \in \O_0\,:\,  x\in C[\z,\b](\o)\}$.

The terms $A_1$, $A_2$ and $A_3$ are similar to the terms $I_N^{(1)}(\b)$, $I_N^{(2)}(\b)$ and $I_N^{(3)}(\b)$ in the arxiv version of  \cite[Section~3]{FM}, respectively
\footnote{We point out an error in  the proof of \cite[Prop.~1]{FM}.  \cite[Eq.~(3.15)]{FM}   is valid only  by inserting in $I^{(1)}_N(\b)$ the characteristic function of the event $\{|C_0^\b(\xi)|\leq N\}$. One has then to treat another contribution given by $I^{(1)}_N(\b)$ with inside  the     characteristic function of the event $\{|C_0^\b(\xi)|> N\}$. This contribution is treated in the arxiv version of  \cite[Prop.~1]{FM}, and also in our proof of Lemma \ref{santo} (in \cite{FM} one  has  to use that $\nabla_x f_N^\b(\xi)=0$ in this case).}. By arguments similar to the ones  used there, one gets the following result:   
\begin{Lemma}\label{santo}  Fix  $\d\in (0,1)$ and a pair of conjugate exponents $p,q\in (1,+\infty)$.  Let $\bbP$ be a stationary marked SPP \magenta{with finite and positive intensity $\rho$}
satisfying \eqref{non_periodico} and  \magenta{$\bbE\bigl[ \widehat{\o}([0,1]^d) ^{ q+1}\bigr]<+\infty$}.
Then, for a suitable constant  $c=c( \d,p)>0$, it holds 
\be\label{francesco} D(\b)_{1,1} \leq c n^{-1/p}\bbE_0\big[ {\rm diam}\, C[\z,\b]\big]^{1/p} +
c
e^{-(1-\d)\z} (1+ \bbE_0 \bigl[ ( {\rm diam}\, C[\z,\b])^{2p}  \bigr] ^{1/p} \magenta{)}\,. \
\en
\end{Lemma}
We sketch the proof (for completeness and due an  error in the proof of \cite[Prop.~1]{FM} already commented).
\begin{proof}
 When   $H_x $ occurs, we have    
$
C[\z,\b](\o)=C[\z,\b](x,\o)= C[\z, \b](\t_x \o) +x$
and therefore $
{\rm diam}\, C[\z,\b](\t_x\o) =  {\rm diam}\, C[\z,\b](\o) $. Hence, if in addition   ${\rm diam}\, C[\z,\b](\o) \leq n$, we get $x_1-\nabla_x f_{n}(\o) =0$, thus implying that $A_1=0$. On the other hand, if $H_x$ occurs and ${\rm diam}\, C[\z,\b](\o) > n$,  then we  have $f_n(\o)=f_n(\t_x\o)=0$, hence  $x_1-\nabla_x f_{n}(\o) =x_1$
and by H\"older's inequality 
 $A_2 \leq  \bbE_0 \big[ \bigl( \sum_{x\in \widehat \o}  e^{-|x|} x_1^2 \bigr)^q \big]^{1/q} 
 \vpz\bigl( {\rm diam}\, C[\z,\b] > n  \bigr)^{1/p}$ (the last \magenta{probability} can be further bounded by Markov inequality). \magenta{At the end we will bound  $\bbE_0 \big[ \bigl( \sum_{x\in \widehat \o}  e^{-|x|} x_1^2 \bigr)^q \big]$.}

 When $H_x^c$ occurs and $x\in \widehat{\o}\setminus\{0\}$, it must be   $c_{0,x}(\o,\b) < e^{-\z}$ (otherwise $\{0,x\}$ would be an edge of $\cG[\z,\b](\o)$). Since we always have $c_{0,x}(\o,\b) \leq e^{-|x|}$, we conclude that $c_{0,x}(\o,\b) \leq e^{-(1-\d)\z}e^{-\d|x|}$.
 As a consequence   we can bound
$
A_3 \leq  3 C_1+3C_2+3C_3$, where 
$C_1:=e^{-(1-\d)\z}\bbE_0 \big[ \sum_{x\in \widehat \o} e^{-\d|x|}  |x_1|^2\big]$, $C_2:=e^{-(1-\d)\z} \bbE_0 \big[ \sum_{x\in \widehat \o}  e^{-\d|x|} f_n(\t_x\o) ^2  \big]$ and 
$ C_3:=e^{-(1-\d)\z} \bbE_0 \big[ \sum_{x\in \widehat \o}  e^{-\d|x|} f_n(\o) ^2  \big]$. By \eqref{non_periodico} (and its equivalent version \eqref{non_periodico_bis}) and \cite[Lemma~1]{FM}--(i), $C_2=C_3$. To bound  $C_3$ we use H\"older's inequality and  that $|f_n|\leq  {\rm diam}\, C[\z,\b]$. We then   get
$C_3\leq e^{-(1-\d)\z}
\bbE_0 \big[ \bigl( \sum_{x\in \widehat \o}  e^{-\d|x|}\bigr) ^q  \big]^{1/q}
   \bbE_0 \bigl[ ( {\rm diam}\, C[\z,\b])^{2 p}  \bigr] ^{1/p}
   $.
   
\magenta{To conclude we need to bound by some $(\d,p)$--dependent constant  the expectations  $\bbE_0 \big[ \bigl( \sum_{x\in \widehat \o}  e^{-|x|} x_1^2 \bigr)^q \big] $,
      $\bbE_0 \big[  \sum_{x\in \widehat \o}  e^{-|x|} x_1^2  \big] $
     and $\bbE_0 \big[ \bigl( \sum_{x\in \widehat \o}  e^{-\d|x|}\bigr) ^q  \big]$  appearing above.
     Let us set $\D_k:=k+ [0,1]^d$. Trivially it is enough to bound $\bbE_0 \big[ \bigl( \sum_{k\in \bbZ^d}  e^{-\frac{\d}{2}|k|}\widehat{\o}(\D_k)  \bigr) ^q  \big]$, which by H\"older's inequality is bounded by
     \be\label{tempo25}\Big( \sum_{k\in \bbZ^d} e^{-\frac{\d}{4} |k|p}\Big)^{q/p}
     \bbE_0 \Big[ \sum_{k\in \bbZ^d}  e^{-\frac{\d}{4}|k|q}\widehat{\o}(\D_k) ^q   \Big]\,.
     \en
   By Campbell's identity \eqref{campanello} with $f(x,\o):=\mathds{1}(x\in [0,1]^d) \widehat{\o}(\D_k) ^q $, the expectation in \eqref{tempo25} is bounded by 
$\rho^{-1}\sum_{k\in\bbZ^d} e^{-\frac{\d}{4}|k|q} \bbE\Big[ \widehat{\o}(\D_0) \widehat{\o}( \D'_k)^q\Big]$ where $\D'_k:=k+[-1,2]^d$ ($\bbE[\cdot]$ being the expectation w.r.t.~$\bbP$). We can bound  $\bbE\big[ \widehat{\o}(\D_0) \widehat{\o}( \D'_k)^q\big]$ uniformly in $k$ by using that $xy^q\leq x^{q+1}+ y^{q+1}$ for $x,y\geq 0 $, the stationarity of $\bbP$ and the assumption that $\bbE\big[ \widehat{\o}([0,1]^d)^{q+1}\big]<+\infty$. As a consequence, the  expectation in \eqref{tempo25}  is bounded  by some  constant determined by $\d$ and $p$.}
\end{proof}

\begin{Remark}\label{bischero} Instead of the above analysis, we could have  used the bound (3.9) in \cite[Prop.~1]{FM}, applied by taking as set   \cite[Eq.~(3.1)]{FM} the edge set  of $\cG [\z,\b]$ and $\z=\ell(\b)$ for a suitable function $\ell(\b)$. On the other hand, we found that fixing now $\ell(\b)$ and introducing the notation of \cite{FM} would have been less trasparent. Note that, in any case, the final choice of the test functions differs from the one in the proof of \cite[Theorem~1]{FM} (see in particular \cite[p.~276]{FM}).
  \end{Remark}

For the next result recall  the definition \eqref{box_pokemon}  of the configuration $\o _{\z,\b}\in\O$     for $\o \in \O$  and $\z,\b>0$, i.e. 
\begin{equation*}
\begin{split}
 \o_{\z,\b}:& =\big\{ \big ( x / \z, \b E_x/\z \big) \,: \,  x\in \widehat{\o}_{\z/\b} \big \}=\big\{ \big ( x/\z, \b E_x/\z\big) \,: \, x\in\widehat \o  \,, |E_x| \leq \z /\b \big\} \,.
\end{split}
\end{equation*}
 Recall that $\o_{\z, \b}$ equals   \eqref{gandhi} 
when $\z=\ell(\b)$. \rosso{Recall also that $\nu(\g)=\nu([-\g,\g])$ and that  $\nu_{\star,\g}$ is the law of  $X/\g$ conditioned to the event $|X/\g|\leq 1$ when $X$ is a random variable with law $\nu$ (cf. respectively \eqref{alloro} and \eqref{sette2}).}

 \begin{Lemma}\label{serafico} Let $\bbP$  be a stationary marked SPP with finite and positive intensity \orc{$\rho$}.  Let  $\z,\b>0$ and let  $\bar \bbP$ be the law of the configuration  $\o_{\z,\b}$   when $\o$ is sampled according to $\bbP$. 
Suppose that $\vpz(\cD)>0$, where
\[ \cD:=\{\o\in \O_0\,:\,  0
 \in \widehat{\o}_{\z/\b} \}=\{\o\in \O_0\,:\,   |E_0|\leq \z/\b\}  \,.
 \] Then:
\begin{itemize}
\item[(i)]  $\bar\bbP$ is a stationary marked  SPP with finite and positive  intensity; 
\item[(ii)] for any  nonnegative Borel function $h$ on $\bbR$  with $h(0)=0$ it holds 
 \be\label{zoomo}
\bbE_0 \left[ \, h \left({\rm diam}\, C[ \z, \b]  \right) \,\right]=  \vpz (\cD)    \bar \bbE_0 \left[ \, h \left(\z\, {\rm diam}\,C[ 1, 1] \right)\, \right] 
  \en
  where $\bar\bbE_0$ is  the expectation w.r.t. the  Palm distribution  $\bvpz$ associated to $\bar \bbP$.
  \end{itemize}
  
  Moreover, if $\bbP=\hat\bbP_\nu$,  then $\vpz(\cD)=\nu( \z/\b)$ and $\bar \bbP$ equals 
   $(\hat \bbP_p\circ \Psi_{\z}^{-1})_{\nu_{\star,  \z/\b}}$ \rot{with $p:=\nu(\z/\b)$}. If in addition $\bbP={\rm PPP}[\rho,\nu]$, then $\bar \bbP$ equals $ {\rm PPP}[\rho \nu(\z/\b)\z^d   ,\nu_{\star,  \z/\b}] $.
   
      \end{Lemma}
 Note that by Item (i) the Palm distribution $\bvpz$ associated to $\bar\bbP$ and appearing in Item (ii)  is well defined.
 \begin{proof}
We first observe that 
 \be\label{ho_sete_78}
  \begin{split}
   \bbE_0 \left[ h\bigl( {\rm diam}\,  C[\z,\b] \bigr)    \right]& =
     \bbE_0 \left[ h\bigl( {\rm diam}\,  C[\z,\b] \bigr)  \rosso{\mathds{1}_\cD} \right]\\
     & = \vpz(\cD)  \bbE_0 \left[ h\bigl( {\rm diam}\, C[\z,\b] \,\bigr)  |\, \cD \right] \,.
     \end{split} 
     \en
     Indeed, the first identity follows from the hypothesis that $h(0)=0$ and the fact that  $ C[\z,\b] (\o) =\{0\}$ if   $|E_0 | >\z/\b$  (see \orc{Lemma}~\ref{fadeev_new}--(i)).
The second identity is just a  conditioning, which is allowed by the hypothesis that \orc{$\vpz(\cD)>0$}.

%
We call 
 $\vvv{\bbQ}$ the law of 
  $\o _{\z/\b}$  when $\o$ is sampled according to $\bbP$. By Lemma 
  \ref{quinoa}  $\vvv{\bbQ}$ is a stationary marked  SPP with finite and positive  intensity. We call $\vqz$ the Palm distribution associated to $\vvv{\bbQ}$.
 By  Lemma \ref{quinoa} $\vqz$ equals the law of $\o _{\z/\b}$ when $\o$ is sampled according to $\vpz(\cdot | \orc{\cD} )$. Since in addition (by \orc{Lemma}~\ref{fadeev_new}--(i))  $\text{diam}\, C[\z,\b] (\o)= \text{diam}\, C[\z,\b] (\o_{\z/\b})$ if $\o\in \cD$, we conclude that 
  \be \label{ho_sete_1} 
 \bbE_0 \left[ h\bigl( \text{diam}\, C[\z,\b]  \bigr) | \cD \right] =\int_{\O_0} 
  d\vqz(\o) 
 h \left( \text{diam}\, C[\z ,\b](\o)\right)\,.
 \en 
Note that  $\bar \bbP$ is the image of $ \vvv{\bbQ}$ when applying the point-mark transformation $(x,E_x)\mapsto  \bigl ( x/\z, \b E_x/\z \bigr) $. It then follows that 
$\bar \bbP$ is a stationary SPP with finite and positive  intensity, and that 
$\bvpz$ is the image of $ \vqz$ by the above transformation. To conclude we observe that, by \orc{Lemma} \ref{fadeev_new}, 
the bijection 
\be\label{vettorino}
\o_{\z/\b}=\{(x,E_x)\,:\, x\in \widehat{\o} _{\z/\b} \}\mapsto \{(x/\z, \b E_x/\z)\,:\,x\in \widehat{\o} _{\z/\b}\} \rot{=}\o_{\z,\b}
\en
induces naturally a bijection between $\cG[\z,\b](\o_{\z/\b})$ and $\cG[1,1](\o_{\z,\b})$.
 By combining this bijection  with the observation that  $\bvpz$ is the image of $ \vqz$ by the  transformation $(x,E_x)\mapsto  \bigl ( x/\z, \b E_x/\z \bigr) $, we then conclude that 
 \be\label{smeraldo} 
 \int_{\O_0} 
  d\vqz\orc{(\o)}
 h \left( \text{diam}\, C[\z ,\b](\o) \right)
 = \int_{\O_0} d\bvpz(\o) 
 h \left(\z \text{diam}\, C[1 ,1](\o) \right)
 \,.
\en
By combining \eqref{ho_sete_78}, \eqref{ho_sete_1} 
and \eqref{smeraldo}  we then get  
 \eqref{zoomo}.
 
  We now move to the last  statement. Let $\bbP:=\hat\bbP_\nu$. Then by  Campbell's formula \eqref{marlena}, by conditioning on $\widehat{\o}$  and by the definition of $\nu$--randomization,  we have 
 \[
 \vpz ( \cD)
 =  \frac{1}{\orc{\rho}}\bbE \Big[ \bbE\big[ \sum _{x\in \widehat \o \cap [0,1]^d}  \mathds{1} ( \t_x \o \in \cD)\,|\, \widehat \o \big]\Big]= \frac{\nu (\z/\b) }{\orc{\rho}}\bbE \big[ \widehat{\o} \big( [0,1]^d\bigr) \big]= 
 \nu (\z/\b) \,. \]
This proves that $ \vpz ( \cD)=\nu(\z/\b)$. The rest of the final statement follows  
   from \orc{Lemma \ref{fadeev_250}}.   \end{proof}

\verde{As a combination of Lemmas~\ref{santo} and \ref{serafico} we get the following:}
\begin{proof}[Proof of Theorem~\ref{supersanto}] \verde{Fix $\d\in (0,1)$. Due to  Lemma~\ref{serafico} with $\z:=\ell(\b)$ and since trivially $\bar{\bbP}^0(\cD)\leq 1$ there,  the bound \eqref{francesco} implies 
\be\label{francescobis} 
\begin{split}
D(\b)_{1,1} & \leq c n^{-1/p}\ell(\b)^{1/p} {\bbE}^\b_0\big[ {\rm diam}\, C[1,1]\big]^{1/p} \\
&+
c
e^{-(1-\d)\ell(\b)} (1+\ell(\b)^{2} {\bbE}^\b_0 \bigl[ ( {\rm diam}\, C[1,1])^{2p}  \bigr] ^{1/p})\,. 
\end{split}
\en
Due to \eqref{limitino77}, for $\b$ large both the expectations in the r.h.s. of \eqref{francescobis} are finite. We first take the limsup as $n\to+\infty$ to remove the first term in the r.h.s. Then we get
$D(\b)_{1,1} \leq c
e^{-(1-\d)\ell(\b)} (1+\ell(\b)^{2} {\bbE}^\b_0 \bigl[ ( {\rm diam}\, C[1,1])^{2p}  \bigr] ^{1/p})$.
 By using \eqref{limitino77} and  the arbitrariness of $\d$ we get \eqref{pingpong}.}
\end{proof}

\subsection{Proof of Theorem \ref{parataUB}}\label{sec_proof_parataUB}
 \verde{By Fact~\ref{teo_eff_cond} it is enough to prove \eqref{preUB}.  We apply  Theorem~\ref{supersanto} by taking $\bbP:={\rm PPP}[\rho,\nu]$,   $\nu \in \cV^+_\a\cup  \cV_\a$ and $\tilde C:= C_\nu$.  Since by this choice  \eqref{pingpong} coincides with \eqref{preUB}, it remains to verify the assumptions of 
Theorem~\ref{supersanto}  in the present setting.
 Due to  Proposition~\ref{paradiso_PPP}, for $\b $ large enough to have $\ell(\b)\leq \min\{C_\nu, \e_\nu\}\b$,   the law $ \bbP^\b$ of $\o_{\ell(\b),\b}$ is given by \eqref{zoomo4}.
Then  all moments of  $\widehat{\o}([0,1]^d)$ are finite for the  marked PPP \eqref{zoomo4} (due to the analogous property for homogeneous PPPs on $\bbR^d$). On the other hand,  \eqref{gioia77} is an immediate consequence of the following percolation result which assures that  ${\rm diam}\,C [1, 1] $ has finite moments w.r.t. the  marked PPP \eqref{zoomo4}.}

\begin{Fact}[A.~Faggionato, A.H.~Mimun \cite{FagMim1}]\label{patatine} Let \rot{$\bbP^0$} be the Palm distribution associated to $\rot{\bbP}:= {\rm PPP}[\l,\nu_{1,  \a}]$ with $\l<\l_c (\a)$ (subcritical case).
Then
there exists a positive constant $c=c(\l,\a)$  such that 
\be
\rot{\bbP^0}\left( {\rm diam}\,C [1, 1] > n   \right) \leq e^{- c \,n } \qquad 
 \forall n \in \bbN\,.
 \en
 The same result holds if we replace $\nu_{1,  \a}, \l_c (\a)$ with $\nu^+_{1,  \a}, \l^+_c (\a)$, respectively.
 \end{Fact}
The above fact follows by applying \cite[Thm.~1.4]{FagMim1} together with the observation that the critical intensities   $\l_c(\a)$, $\l_c^+(\a)$ coincide with the one in \cite[Eq.~(1.5)]{FagMim1} as proved in Appendix \ref{liberato} (see also the discussion above Theorem 1.4 in \cite{FagMim1} for the check of the assumptions in our case).

\section{\verde{Proof of Theorems~\ref{parataLB} and~\ref{santantonio}  (lower bound in Mott's law)}}\label{sec_LB}

\verde{In this section we prove Theorem~\ref{santantonio}. By applying it, in Section~\ref{sec_proof_parataLB} we prove Theorem~\ref{parataLB}}. 

We can extend the definitions introduced in  Section \ref{fumata} to general  graphs. More precisely, 
given a weighted non--oriented graph $\bbG=(\bbV,\bbE)$ with weight function $C: \bbE\to [0,+\infty)$ and vertex set $\bbV \subset \bbR^d$ and given $\ell>0$, we consider the resistor network on $S_\ell$ with nodes $x\in \bbV \cap S_\ell$, edges $\{x,y\}$ with $x,y\in \bbV \cap S_\ell$ and $\{x,y\}\cap \L_\ell \not = \emptyset$ and electrical conductivity of the edge  $\{x,y\}$ given by $C(x,y)$.

As in the derivation of \eqref{papero}, one can easily \rot{prove} Dirichlet's  principle (cf.~\cite[Exercise~1.3.11]{DS}):
\be\label{papero_bis}
\s_\ell (\bbG,C )=\inf \left\{ \cD_\ell(u) \,\big{|} \, u: \bbV \cap S_\ell \to [0,1]\,, u_{|\bbV \cap S_\ell^-} \equiv 1
\,, \; u_{|\bbV \cap S_\ell^+}\equiv0  \right\}\,,
\en
where 
\begin{align*}
&  \cD_\ell(u) := \sum _{\{x,y\} \in \bbE_\ell} C(x,y) \left[ u(y)- u(x) \right]^2\,,\\
& \bbE_\ell:=\bigl\{\,\{x,y\}\in \bbE\,:\, x \in  \bbV \cap \L_\ell\,,\; y \in \bbV \cap S_\ell, \; x\not=y\bigr\}\,.
\end{align*}
As a consequence,   $\s_\ell (\bbG,C)$ satisfies  Rayleigh's monotonicity law (cf.~\cite[Section~1.4]{DS}): 
\be\label{rayleigh_mon_law}
\bbG\supset \bbG' \text{ and } C\geq C'\;\;\Rightarrow \;\; \s_\ell(\bbG, C)\geq \s_\ell(\bbG', C')\,.
\en


\medskip

Recall that, given a graph $\bbG=(\bbV,\bbE)$ with vertex set  $\bbV\subset \bbR^d$,   $\cN_\ell(\bbG)$ denotes  the  maximal number of  vertex-disjoint LR crossings of $\L_\ell$ in $\bbG$ (cf. Definition \ref{def_LR}).

\begin{Lemma}\label{lemma_LB_chiave}
Given a graph $\bbG=(\bbV,\bbE)$ with vertex set  $\bbV\subset \bbR^d$, we have\be\label{LB_chiave}
\s_\ell (\bbG, 1) \geq \frac{\cN_\ell(\bbG) ^2 }{2\cN_\ell(\bbG)+ |\bbV \cap \L_\ell|}\geq \frac{\cN_\ell(\bbG) ^2 }{3 |\bbV \cap \L_\ell|} \,.
\en
\end{Lemma}

\begin{proof} Let $n:= \cN_\ell (\bbG)$.
Let $u : \bbV\cap S_\ell \to [0,1]$ be a function with $u_{ |\bbV \cap S_\ell^-}\equiv 1$ and $u_{|\bbV \cap S_\ell^+}\equiv 0$. We set $\nabla_{x,y}u :=u(y)-u(x)$.  We fix a family  $\p_1,\ldots,\p_n$ of  vertex--disjoint LR crossings of $\L_\ell$ in $\bbG$. Recall that $\pi_i $ is a sequence $(x_1, x_2, \dots, x_k)$ of $k\geq 3$ distinct points in $\bbV$ such that 
$\{x_r,x_{r+1}\}\in \bbE$ for all $r\in\{1,\dots, k-1\}$, 
$x_1 \in \cS_\ell^-$, $x_k \in \cS_\ell^+$ and $x_r\in \L_\ell  $ for $1< r<k$. Below, we write $\{x,y\}\in \pi_i$ if $\{x,y\}$ is an edge of $\pi_i$, i.e. $\{x,y\}=\{x_r, x_{r+1}\}$ for some some $r$. We also set $|\p_i|:= k$.

As the $\p_i$'s are vertex-disjoint, they are also edge-disjoint. Hence, by Rayleigh's monotonicity law (cf.~\eqref{rayleigh_mon_law}),  we can bound 
\begin{equation}
\label{low1}
\s_\ell (\bbG,1) \geq \sum_{\{x,y\}\in \bbE_\ell}  (\nabla_{x,y} u)^2 \geq\sum_{i=1}^n\sum_{\{x,y\}\in \p_i}(\nabla_{x,y} u)^2 
\,.
\end{equation}
By Cauchy-Schwarz's inequality we obtain  $1=\Big(\sum_{\{x,y\}\in \p_i}\nabla_{x,y} u\Big)^2$ is upper bounded by $ |\p_i|\sum_{\{x,y\}\in \p_i} (\nabla_{x,y} u)^2$.
Hence the r.h.s. in \eqref{low1} can be bounded from below as
\begin{equation}
\label{low2}
\begin{split}
\sum_{i=1}^n\sum_{\{x,y\}\in \p_i} (\nabla_{x,y} u)^2\geq  \sum_{i=1}^n\frac{1}{|\p_i|}.
\end{split}
\end{equation}
By Jensen's inequality  applied to the convex function $ x \mapsto 1/x$ for $x>0$,  we can bound from below the r.h.s. of \eqref{low2} as
\begin{equation}
\label{low3}
\begin{split}
 \sum_{i=1}^n\frac{1}{|\p_i|}&=n \cdot\frac{1}{n}\sum_{i=1}^n\frac{1}{|\p_i|}\geq n\left(\frac{1}{n}\sum_{i=1}^n|\p_i|\right)^{-1}
=\frac{n^2 }{\sum_{i=1}^n|\p_i|}
\end{split}
\end{equation}
(the above use of Jensen's inequality is not new,  see e.g.~the proof of  \cite[Prop.~3.2]{CC}).
Recall  that the $\pi_i$'s are vertex-disjoint  and  all \rot{vertices} of each $\pi_i$ are in $\bbV \cap \L_\ell$ apart \rot{from} the extreme ones.
As a consequence, we can bound  $\sum_{i=1}^n|\p_i|\leq  2n +| \bbV \cap \L_\ell|$.
 Combining this bound with \eqref{low1}, \eqref{low2} and \eqref{low3} we get  the first bound in \eqref{LB_chiave}. For the second one it is enough to observe  that the $\pi_i$'s are vertex-disjoint and that each $\pi_i$ has at least one point in $\bbV \cap \L_\ell$.
\end{proof}



\rot{The following result gives a lower bound of the conductivity $\s_\ell(\o, \b)$ by means of thinning and rescaling and it is based  on Lemma~\ref{fadeev_new}:}
\begin{Lemma}\label{lasagne25}
Let \magenta{$\z,\b,\ell>0$}. Then  for any $\o\in \O$ it holds 
\be\label{pasta_rossa}
\ell^{2-d}\s_\ell(\o, \b)\geq  e^{-\z} \z^{2-d} L ^{2-d} \s_{L} ( \cG[1,1](\o_{\z,\b}  ),1)\,, \qquad  L:= \ell/\z\,,
\en
where $\o_{\z,\b}$ is defined in \eqref{box_pokemon}.
\end{Lemma}
\begin{proof}
We denote by $\cE[\z,\b]$ the edge set of the graph $\cG[\z,\b]$. By definition of $\cG[\z,\b]$ we have
\be \label{arca0}
c_{x,y}(\o, \b) \geq e^{-\z} \mathds{1} \bigl( \{x,y\} \in \cE[\z,\b](\o)\bigr)\,.
\en
We claim that 
\be \label{arca1}
\s_\ell\bigl(\o, \b\bigr) \geq e^{-\z} \s_\ell \bigl(\cG[\z,\b](\o) , 1\bigr)=   e^{-\z} \s_\ell \bigl(\cG[\z,\b](\o_{\z/\b} ) , 1\bigr) \,.
\en
Indeed, the  bound follows from Rayleigh's monotonicity law and \eqref{arca0}, while the identity follows from   Lemma~\ref{fadeev_new}--(i): points of $\widehat{\o}\setminus \widehat{\o}_{\z/\b}$ are isolated in the graph $\cG[\z,\b] $ and therefore they do not contribute to the conductivity. Recall now the  bijection between $\cG[\z,\b](\o_{\z/\b})$ and $\cG[1,1](\o_{\z,\b})$ given by \orc{Lemma} \ref{fadeev_new}--(ii). This bijection and \eqref{papero_bis} imply that 
$\s_\ell \bigl(\cG[\z,\b](\o_{\z/\b} ) , 1\bigr)= \s_{\ell/\z} \bigl(\cG[1,1](\o_{\z,\b} ) , 1\bigr)$. 
\orc{By combining  this identity with} \eqref{arca1}  we get \eqref{pasta_rossa}.
\end{proof}

By using Lemmas~\ref{lemma_LB_chiave} and~\ref{lasagne25} we can conclude the proof of Theorem~\ref{santantonio}:
\begin{proof}[Proof of Theorem~\ref{santantonio}]
\verde{By Lemmas  \ref{lemma_LB_chiave}  and  \ref{lasagne25}  with $\z=\ell(\b)$ (do not confuse $\ell$ and $\ell(\b)$), setting  $ L:= \ell /\z=\ell/\ell(\b)$ we have 
 \be\label{pasta_rossa_bis}
 \begin{split}
\ell^{2-d}\s_\ell(\o, \b)& \geq  e^{-\ell(\b) } \ell(\b) ^{2-d} L^{2-d} \s_{L} ( \cG[1,1](\o_{\ell(\b),\b}  ),1)\\
& \geq  e^{-\ell(\b) }\ell(\b) ^{2-d}L^{2-d} \frac{\cN_L\left( \cG[1,1](\o_{\ell(\b) ,\b}  )\right)^2}{ 3 | \widehat{\o_{\ell(\b) ,\b} } \cap \L_L|}\,.
\end{split}
\en}

\verde{We first take the limit $\ell \to +\infty$ thinking \rot{of}  $\b$ as fixed (trivially, $L\to +\infty$).  
 By the ergodic theorem we have that $(2L)^{-d} | \widehat{\o_{\ell(\b) ,\b} } \cap \L_L|$ converges to the intensity $\rho(\b)$ of $ \bbP^\b$   as $\ell \to +\infty$.
We then get that $\bbP$--a.s. 
\begin{multline}\label{pasta_bianca}
\limsup_{\ell\to +\infty} L ^{2-d} \frac{\cN_L\left( \cG[1,1](\o_{\ell(\b) ,\b}  )\right)^2}{ 3 | \widehat{\o_{\ell(\b) ,\b} } \cap \L_L|}\\
=\frac{1}{2^{d}}
\limsup_{\ell\to +\infty}\left(\frac{\cN_L\left( \cG[1,1](\o_{\ell(\b) ,\b}  )\right)}{L^{d-1}}\right)^2
\frac{(2L)^d}{3 | \widehat{\o_{\ell(\b) ,\b} } \cap \L_L|}
\\=\frac{1}{3 \rho (\b)2^d}\limsup_{\ell\to +\infty}\left(\frac{\cN_L\left( \cG[1,1](\o_{\ell(\b) ,\b}  )\right)}{L^{d-1}}\right)^2=:\frac{ C(\b)^2}{3 \rho (\b)2^d}
\,.
\end{multline}
We have (cf.~\eqref{gandhi}) that $\rho(\b)$ is upper bounded by the intensity of the SPP $\{ x/\ell(\b): x\in \widehat{\o}\}$ when $\o$ is sampled according to $\bbP$. Hence $\rho(\b) \leq \ell(\b)^d \rho$.
By combining Fact~\ref{teo_eff_cond} with  \eqref{pasta_rossa_bis} and \eqref{pasta_bianca}, we get 
\be\label{rispetto}
\s(\b)=\lim_{\ell\to +\infty}\ell^{2-d}\s_\ell(\o, \b) \geq e^{-\ell(\b) } \ell(\b) ^{2-d}  C(\b)^2 (3 \rho \ell(\b)^d  2^d)^{-1} \,.
\en
Since $\ell(\b):= (\l  \tilde C^{\a+1}/\rho)^{\frac{1}{\a+1+d}} \b ^{\frac{\a+1}{\a+1+d}}$, by taking 
$\liminf _{\b \uparrow \infty}  \b^{-\frac{\a+1}{\a+1+d}} \ln(\cdot) $ of both sides of \eqref{rispetto} and using \eqref{squid_game_3}, we get  \eqref{polpetta}.}
\end{proof}

\subsection{Proof of Theorem~\ref{parataLB}}\label{sec_proof_parataLB}
\verde{By Fact~\ref{teo_eff_cond} it is enough to prove \eqref{LB}. We apply Theorem~\ref{santantonio} with $\bbP:={\rm PPP}[\rho, \nu]$,  $\tilde C:=C_\nu$ and $\l<\l_c(\a)$. To this aim we comment the validity of \eqref{squid_game_33} (the other assumptions can be easily verified). We first observe that  $\bbP^\b$ is given by \eqref{zoomo4} for $\b$ large (cf.~Proposition~\ref{paradiso_PPP}). As a consequence \eqref{squid_game_3} follows from Fact~\ref{cortocircuito} when $\nu\in \cV_\a^+$ and by  our hypothesis of good statistics of LR crossings for $\nu\in\cV_\a$ (which  is a claimed result in \cite{FH}). Having \eqref{polpetta}, to conclude it is enough to invoke the arbitrariness of $\l<\l_c(\a)$.}

\appendix

\section{Critical values $\z_c (\b,\r,\nu)$, $\l_c(\a)$ and  $\l^+_c(\a)$}\label{guizzo}
%

Given $r>0$ and  $\hat \o \in \hat \O$, the Boolean model $G_r(\hat \o)$ with \rot{deterministic} radius $r$ is the graph with vertex set $\hat \o$ and edges $\{x,y\}$ with $0<|x-y|\leq 2 r$. When $\hat \o$ is sampled according to ${\rm PPP}(\l)$, this random graph is called Poisson Boolean model. In this case 
it is known that, for a suitable $\l_c\in (0,+\infty)$, $G_r(\hat \o)$  a.s. does not percolate for $\l<\l_c$ and $G_r(\hat \o)$  a.s.   percolates for $\l>\l_c$  (cf.~\cite{MR}).

 \rosso{As in \eqref{flautino1}, given $\o\in \O$ and $\g>0$, we set $\widehat{\o}_\g:=\{x\in \widehat{\o}\,:\, |E_x|\leq \g\}$.}

\subsection{Critical threshold $\z_c (\b,\r,\nu)$}\label{marino} We now derive the existence of $\z_c (\b,\r,\nu)$ in $ (0,+\infty)$ satisfying \eqref{critico}.  Sample $\o$ according to ${\rm PPP}[\r,\nu]$.
 Since  $\cG [\z,\b](\o)$ has vertex set $\widehat \o$ and edges $\{x,y\}$ with $x\not=y$ and $|x-y|+\b( |E_x|+|E_y|+|E_x-E_y|) \leq \z$,  $\cG [\z,\b](\o)$
 is contained in $G_{\z/2}(\widehat{\o})$ and contains  $G_{\z/10}(\widehat{\o}_{\z/5\b})$. 
 By varying $\o$ the above random graphs are Poisson Boolean models with deterministic radii $\z/2$, $\z/10$ and intensities $\rho$, $\rho \nu(\z/5\b)$, respectively.
  Moreover, they are mapped by the homothety  $\bbR^d\ni x\mapsto x/\z$ into Poisson Boolean models with deterministic radii $1/2$, $1/10$ and intensities $\rho \z^d $, $\rho \nu(\z/5 \b)\z^d$, respectively. Hence, by the above  discussed phase transition of the Poisson Boolean model, a.s. $G_{\z/2}(\widehat{\o})$ does not percolate for $\z$ small and  a.s.  $G_{\z/10}(\widehat{\o}_{\z/5\b})$ percolates for $\z$ large. 
 This proves that $\theta(\z,\b, \rho, \nu)$ 
 equals $0$  for $\z$ small and $1$ for $\z$ large.  
 Since  $\theta(\cdot,\b, \rho, \nu)$ is trivially non-decreasing and, as already mentioned, has value in $\{0,1\}$, we conclude that there exists  $\z_c (\b,\r,\nu)$ in $ (0,+\infty)$ satisfying \eqref{critico}.

\subsection{Critical intensities $\l_c(\a)$ and  $\l^+_c(\a)$}\label{sfondo} We now prove \orc{the} existence of values $\l_c(\a) $ and $\l_c^+(\a)$ as in Definition  \ref{fiordo}.
Let $\o$ be sampled according to $\text{PPP}[\l, \nu]$, where $\nu=\nu_{1,\a}^+$ or $\nu=\nu_{1,\a}$.
As discussed in Section \ref{marino},  $\cG[1,1](\o)$ is contained in $G_{1/2}(\widehat{\o})$ and contains  $G_{1/10}(\widehat{\o}_{1/5})$, which are 
  Poisson Boolean models with deterministic radii $1/2$ and $1/10$  and intensities $\l$ and  $\l \nu(1/5)$ respectively. Since $\nu(1/5)>0$, the conclusion follows from the   phase transition of the Poisson Boolean model and the property that  \orc{$ \theta(1,1, \cdot, \nu)$} is non-decreasing.

%

\subsection{Comparison of $\l_c(\a)$ and  $\l^+_c(\a)$ with the critical intensity  in \cite{FagMim1}} \label{liberato}
In order to apply the results of \cite{FagMim1} we show that 
\be\label{rinoceronte}
 {\rm PPP}^0(\l, \nu_{1,\a})\left( 0\leftrightarrow \infty\text{ in } \cG[1,1]\right) 
\begin{cases}
>0 &\text{ if } \l > \l_c(\a)\,,\\
=0 &\text{ if } \l < \l_c(\a)\,.
\end{cases} \en
Above ${\rm PPP}^0(\l, \nu_{1,\a})$ denotes the Palm distribution of  ${\rm PPP}(\l, \nu_{1,\a})$.
Moreover, a similar expression (with similar derivation) holds when replacing  $\nu_{1,\a}$ and $\l_c(\a)$ by $\nu^+_{1,\a}$ and $ \l^+_c(\a)$. As a consequence,  $\l_c(\a)$, $\l_c^+(\a)$ coincide with the critical intensity in \cite[Eq.~(1.5)]{FagMim1}. 

To shorten the notation let $\bbP:=  {\rm PPP}(\l, \nu_{1,\a})$, $\cA:=\{0 \leftrightarrow \infty\}$  and  $\cC:=\{ \cG[1,1](\o)\text{ percolates}\}$.  When 
$\l < \l_c(\a)$ we have  $\bbP(\cC)=0$. Since $\cC$ is translation invariant, by \cite[Lemma~7.1]{Fhom1}  we conclude that $\vpz(\cC)=0$, thus implying that $\vpz(\cA)=0$. If $\l > \l_c(\a)$, then $\bbP(\cC)=1$, thus implying that $\bbP(\cB)>0$ where $\cB:=\{\exists x\in \widehat{\o}\cap[0,1]^d\,:\, x\leftrightarrow \infty\}$ (since $\cC$ is a countable union of translations of $\cB$).
By Campbell's identity \eqref{campanello} with $f(x,\o):=\orc{\mathds{1}_{[0,1]^d}(x) \mathds{1}_\cA (\o )}$ and since $\bbP(\cB)>0$, we get
$
\vpz(\cA)=\frac{1}{\l} \int _\O d\bbP(\o) \sum_{x\in [0,1]^d} \mathds{1}(x \leftrightarrow \infty)\geq \frac{1}{\l}\bbP(\cB)>0$.

\section{Proof of Fact \ref{cortocircuito}}\label{pinco}
To get Fact \ref{cortocircuito}
it is enough \rot{to} apply  \cite[Thm.~1]{FagMim2} with structural function $h(a,b)=1-(|a|+|b|+|a-b|)$, probability  $P$ given by ${\rm PPP}[\l,\nu^+_{1,\a}]$ and distribution $\nu =\nu^+_{1,\a}$. To this aim we need to   check the  validity of the Assumptions (A1),...,(A5) in \cite[Section~2]{FagMim2}.  Since $\nu$ has support $[0,1]$ and for $a,b\in [0,1]$ it holds $h(a,b)=1-2 \max\{a,b\}$, the function $h$ is decreasing in $a,b$ varying in the support of $\nu$. This proves (A5). As discussed before Theorem 1 in \cite{FagMim2}, since $h$ is continuous, $\nu$ has bounded support and (A5)  is satisfied, (A4) is automatically satisfied.
(A1) and (A3) are trivially satisfied. The core is to check (A2), and our proof below is inspired by \cite[Section~8]{FagMim2}.  In our language, we need to show that there exist $\g\in (0,1)$ and \orc{$\l_*<\l $} such that $\cG[\g,1](\o)$ percolates for ${\rm PPP}[\l_*, \nu]$--a.a.~$\o$ \orc{(we recall that the $1$ in $\cG[\g,1](\o)$ is the inverse temperature)}.

By definition of $\l_c^+(\a)$, for any fixed $m> \l_c^+(\a)$ the graph $\cG[1, 1] (\o)$ percolates for  ${\rm PPP}[m,\nu]$--a.a. $\o$. 
Let us fix also  $\g\in (0,1)$. If $X$ is a random variable with law $\nu$, then   $\g X$ has law  $\nu_\g$ because of  the special form of $\nu_{1,\a}^+$ (see Definition \ref{san_benedetto}). Hence the marked SPP given by $\tilde \o:=\{ (\g x, \g E_x)\,:\, x\in \widehat \o\}$  with $\o$ sampled according to ${\rm PPP}[m,\nu]$ has law ${\rm PPP}[ \g^{-d} m, \nu_\g]$. On the other hand, when $\cG[1, 1] (\o)$ percolates, also $\cG[\g, 1](\tilde \o)$ percolates since  the map $(x,E_x) \mapsto (\g x, \g E_x)$ induces a vertex bijection between the two graphs
 and \orc{since, given $x\not= y$ in $\widehat{\o}$,   $\{x,y\}$ is an edge of of $\cG[1, 1] (\o)$ if and only if}  it holds 
$|\g x- \g y| + |\g E_x|+|\g E_y|+|\g E_x- \g E_y| \leq \g $.
Hence, by the above observations, \orc{$\cG[\g, 1] (\o)$} percolates a.s. when \orc{$\o$}  is sampled according to ${\rm PPP}[ \g^{-d} m, \nu_\g]$. \orc{In conclusion, given $m'>0$, $\cG[\g, 1] (\o)$ with $\o$ sampled by ${\rm PPP}[m',\nu_\g]$ percolates a.s. if $\g^d m'>\l_c^+(\a)$.}

Let us fix $\l_*\in (\l_c^+(\a),\l)$.
By Lemma \ref{lemma_2021}, if $\o $ is sampled by ${\rm PPP}[\l_*, \nu]$, then $\o_\g:=\{(x,E_x)\,: \, x\in \widehat\o,\, |E_x|\leq \g\}$
 has law ${\rm PPP} [\nu(\g) \l_*, \nu_\g]$.  \orc{In particular, since $\cG[\g,1](\o)$  contains $\cG[\g,1](\o_\g)$ and by the previous conclusion,
 to prove that  $\cG[\g,1](\o)$ a.s. percolates with $\o$ sampled by ${\rm PPP}[\l_*, \nu]$ it is enough to have  $\g^d \nu(\g) \l_*>\l_c^+(\a)$. As $ \l_*>\l_c^+(\a)$ and $\nu(\g)=\g^{\a+1}$, it is enough to take $\g$ enough close to $1$. This proves the validity of  (A2).}

\smallskip \noindent {\bf Acknowledgments}:  I thank \magenta{the referees for their careful reading and precious corrections and  suggestions. I also thank Muhammad Sahimi  and Hermann Schulz-Baldes for useful discussions concerning Mott v.r.h. and doped semiconductors.}

\magenta{Much of this work was written in the company of my beloved dog Ciak, who was a cherished member of my family and who passed away too soon. This work is dedicated to him, with infinite love and gratitude. }


\end{document}